\documentclass[12pt]{article}
\usepackage{amsmath,amssymb,amsthm}
\usepackage{url}
\usepackage[arrow,matrix,curve]{xy}
\newtheorem{thm}{Theorem}[section]
\newtheorem{lem}[thm]{Lemma}
\newtheorem{prop}[thm]{Proposition}
\newtheorem{cor}[thm]{Corollary}
\newtheorem*{thm*}{Theorem}
\newtheorem*{prop*}{Proposition}
\newtheorem*{cor*}{Corollary}
\theoremstyle{definition}
\newtheorem{ex}[thm]{Example}
\theoremstyle{definition}
\newtheorem{rmk}[thm]{Remark}
\theoremstyle{definition}
\newtheorem{df}{Definition}

\begin{document}

\author{Hanno von Bodecker\footnote{Fakult{\"a}t f{\"u}r Mathematik, Ruhr-Universit{\"a}t Bochum, 44780 Bochum, Germany}}
\title{On the geometry of the $f$-invariant}
\date{}

\maketitle

\begin{abstract}
The $f$-invariant is a higher version of the $e$-invariant that takes values in the divided congruences between modular forms; it can be formulated as an elliptic genus of manifolds with corners of codimension two.

In this thesis, we develop a geometrical interpretation of the $f$-invariant in terms of index theory, thereby providing an analytical link between the stable homotopy groups of the spheres and the arithmetic of modular forms. 
In particular, we are able to establish a formula that allows us to compute the $f$-invariant from a single face. Furthermore, we apply our results to the situation of cartesian products and principal circle bundles,  performing explicit calculations. 
\end{abstract}

\newpage
\thispagestyle{empty}
\enlargethispage*{29cm}
\tableofcontents

\newpage
\addcontentsline{toc}{section}{Introduction}
\section*{Introduction}

The computation of the stable homotopy groups of the sphere, which, by the Pontrjagin-Thom construction, can be interpreted as the bordism groups of framed manifolds, is one of the most fundamental problems in pure mathematics, and the Adams-Novikov spectral sequence (ANSS) serves as a powerful tool to attack this problem, see e.g.\ \cite{rave}. In \cite{Lau00}, the ANSS is interpreted in terms of manifolds with corners, their codimension corresponding to the AN filtration degree; furthermore, an invariant for elements of second AN filtration is proposed. As a follow-up to the degree and the $e$-invariant, this so-called $f$-invariant arises as an elliptic genus of manifolds with corners of codimension two and takes values in the divided congruences between modular forms.

It is well-known that the classical genera of {\em{closed}} manifolds can be understood in terms of index theory, and the seminal work of Atiyah, Patodi, and Singer on index theory on manifolds with {\em{smooth}} boundary \cite{APS1} can be used  to relate the spectral asymmetry of Dirac operators to the $e$-invariant  \cite{APS2}.

The purpose of this thesis is to  show that these powerful ideas also enable us to provide geometrical insight into the $f$-invariant. The outline is as follows: 

Section \ref{prelim} consists of a brief recollection of the relevant background material, namely tangential structures, $\langle k\rangle$-manifolds, characteristic forms, some index theory results, and the Hirzebruch elliptic genus of level $N$, followed by an (admittedly biased) exposition of the index theoretical aspects of the $e$-invariant; lastly, we recall the definition of the $f$-invariant in terms of the relative classes of a suitable $\langle2\rangle$-manifold.

In section \ref{main}, we set out to explore the geometry behind the occurrence of the  divided congruences in the definition of the $f$-invariant. A somewhat surprising result is the following:

\begin{thm*}
Let $X$ be a compact manifold of dimension $2n$, and let $E_1$, $E_2$, and $F$  be hermitian vector bundles over $X$ such that $E_1\oplus E_2\cong TX^{st}$,  and that there is a given trivialization  $\psi:E_1|_{\partial X}\cong\partial X\times \mathbb{C}^k$. Then, for compatible connections, the inhomogeneous combination of modular forms given by
$$\int_X{\widetilde{Ell}(\nabla^{E_1})Ell_0(\nabla^{E_2})ch(\nabla^{F})}$$
has an integral $q$-expansion.
\end{thm*}

Having this result at our disposal, we write down a similar expression on a  $(U,fr)^2$-manifold and determine which reductions have to be made in order to 
retrieve the information that depends only on its corner; we call the resulting geometrical invariant $\check{f}$, and it will serve as a tool for an index theoretical computation of the topological $f$-invariant later on.

Section \ref{calculable?} is devoted to establishing a method of calculating the $f$-invariant in a more analytical fashion, based on $\check{f}$. 
To this end, we are going to introduce the notion of a {\em $(U,fr)^2_f$-manifold}; roughly speaking, these manifolds are families of $(U,fr)$-manifolds parameterized by $(U,fr)$-manifolds. This construction  enables us to reformulate the $f$-invariant in the following form:

\begin{thm*}
Let $Z$ be a $(U,fr)^2_f$-manifold,  and let  $\nabla_1=\pi^*\nabla^{E}$ and $\nabla_2$  be compatible connections on $E_1=\pi^*E$ and $E_2$, respectively.  
Then the $f$-invariant  of its  corner $M$ is given by
$$f(M)\equiv \int_B{\hat{e}_{\Gamma}\widetilde{Ell}(\nabla^{E})},$$
where we defined \textup{(}a de Rham representative of\textup{)} the $e_{\Gamma}$-invariant of  a family:
$$\hat{e}_{\Gamma}\equiv\int_{Z/B}Ell_0(\nabla_2)\mod\left(im(ch:K^{\Gamma}(B)\rightarrow H^{even}(B,\mathbb{Q}[\zeta]))\right)_{dR}.$$
\end{thm*}

This result furnishes a close analogy to the $e$-invariant, and makes possible the use of  `classical'  techniques from index theory to compute the $f$-invariant from a single face:

\begin{cor*}
Let $Z$ be as above. Assume that the kernel of the twisted Dirac family $\eth^{\Gamma}_{\partial X}$ is of constant rank along the fibers. Then the $f$-invariant of the corner is given by
$$\int_B{\left\{\left(\hat{\eta}\left(\eth^{\Gamma}_{\partial X}\right)+{\textstyle{\frac{1}{2}}}ch\left(\nabla^{\ker\eth^{\Gamma}_{\partial X}}\right)+{\textstyle{\int_{Z''/B}{cs}}}\right) \widetilde{Ell}(\nabla^{E})\right\}}.$$
\end{cor*}
Furthermore, we establish a vanishing result:

\begin{thm*}
Let $M$ be the codimension-three corner of a $(U,fr)^3_f$-manifold $Y$. Then the $f$-invariant of $M$ is trivial. 
\end{thm*}

The following sections focus on sample calculations and further simplifications: In section \ref{productsection}, we treat cartesian products, in which case we can actually reduce the computations to the corner itself:

\begin{thm*}
Let $Y_1$, $Y_2$ be odd-dimensional framed manifolds, and let $m(Y_i)$ be {\em any} modular form of weight  $(\dim Y_i +1)/2$ w.r.t.~$\Gamma=\Gamma_1(N)$ such that $m(Y_i)\equiv e_{\mathbb{C}}(Y_i)\mod\mathbb{Z}^{\Gamma}[\![q]\!]$. Then we have
$$f(Y_1\times Y_2)\equiv {m}(Y_1) e_{\mathbb{C}}(Y_2)\equiv-{m}(Y_2) e_{\mathbb{C}}(Y_1).$$
\end{thm*}

Thus, the $f$-invariant of a product is determined by the $e_{\mathbb{C}}$-invariants of the factors, and the latter can be calculated by various  means. We illustrate this result by performing explicit calculations at the level $N=3$, covering a broad  variety of products. 

In section \ref{circlebundlesection}, we turn our attention to principal circle bundles, and our index theoretical approach yields the following: 

\begin{cor*}
Let $L$ be a hermitian line with unitary connection $\nabla^L$ over a $(U,fr)$-manifold $B$ of dimension $2n+2$, and let  $S(L)_|$ be the  framed  circle bundle over $\partial B$. Then we have 
$$f(S(L)_|)\equiv\sum_{k=0}^{n} \frac{B_{k+1}}{(k+1)!}\int_B{\left\{\left(\frac{iF^L}{2\pi}\right)^k\widetilde{Ell}(\nabla^{E})\right\}}.$$
\end{cor*}
 
As an application, we treat the case of principal torus bundles over a framed base. In particular, we perform explicit calculations for the generic situation up to (total) dimension 14, enabling us to determine the necessary and sufficient conditions for non-triviality (at the level $N=3$) in this range.

Finally, we decided to provide a rather extensive appendix: Besides some useful formul\ae, we remind the reader of the theory of modular forms of level $N$, deriving {{explicit}} expressions for $N=3$. Furthermore we expand the Hirzebruch genus at this level, thereby correcting some errors in the appendix of \cite{HBJ}. Moreover, we compile a list of the congruences relevant to the  computations carried out in the main part.  For the sake of completeness, we  also included a derivation of the $\hat\eta$-form in the situation of a principal circle bundle (following \cite{zhang94}).

\subsection*{Acknowledgements} 
First of all, I would like to thank my advisor, Professor Gerd Laures, for introducing me to the fascinating subject of elliptic cohomology and its diverse applications, and for proposing the topic of and supervising the work on this thesis. At the same time, I want to thank Professor Uwe Abresch, who has always been willing to share  his insight and interest in a wide variety of mathematical topics. 
Furthermore, I am grateful to all the members of the topology chair for creating such a pleasant atmosphere. 
 Last but not least, I acknowledge financial support from the DFG within the Graduiertenkolleg 1150 ``Homotopy and Cohomology''.

\section{Preliminaries}\label{prelim}

\subsection{Tangential structures}

We start by reviewing some basic definitions appearing in bordism theories, see e.g.~\cite{cf66}, \cite{stong68}: Let $X$ be a smooth compact manifold of dimension $n$ and consider `the' {\em stable tangent bundle} 
$$TX^{st}=TX\oplus X\times\mathbb{R}^{2k-n},\ 2k\geq n+2;$$
we suppress the dependence on $k$, as it does not matter for our purposes (as long as $2k\geq n+2$).
A {\em stably almost complex structure} on $X$ is a complex structure on $TX^{st}$, i.e.~a linear bundle map $J$ covering the identity and squaring to minus the identity on the fiber. Thus, $TX^{st}$ may be identified with a complex vector bundle $E$ (of rank $k$); consequently, we can define its Chern classes, and there is a preferred orientation on $X$. 
We denote the underlying homotopy class of the stably almost complex structure by $\phi$, which we will refer to as a {\em $U$-structure}; phrased differently, a $U$-structure is a lift (up to homotopy) of the classfying map of the tangent bundle: 
$$\begin{xy}
\xymatrix{
 & & BU\ar[d] \\
 X \ar[rr]^{TX^{st}}
 \ar[urr]^{\phi}& &BO }\\
 \end{xy}$$  
A pair $(X,\phi)$ is called a {\em $U$-manifold}; it will simply be  denoted by $X$ if confusion is unlikely.  Any $U$-structure admits a  `negative',  $-\phi$, and we set $-(X,\phi)=(X,-\phi)$; furthermore, $U$-structures are compatible w.r.t.~taking boundaries, i.e.~$\partial(X,\phi)=(\partial X,\partial\phi)$. We define an equivalence relation called {\em$U$-bordism}, $X_1 \sim_U X_2$, if  there is a $U$-manifold $W$ such that $\partial W\cong  X_1 \sqcup -X_2$. Disjoint union and cartesian product turn the set of $U$-bordism classes into a graded ring, the {\em complex bordism ring}  $\Omega^U_*$.

Similarly, a {\em framing} of $X$ is trivialization of the stable tangent bundle,
$$\psi:TX^{st}\cong X\times\mathbb{R}^{2k};$$
thus, up to homotopy, we have a lift of the classifying map to $EO$:
$$\begin{xy}
\xymatrix{
 & & EO\ar[d] \\
 X \ar[rr]^{TX^{st}}
 \ar[urr] & &BO }\\ 
 \end{xy}$$  
A {\em framed manifold} is a manifold with a homotopy class of trivializations of the stable tangent bundle. Take note that, by pulling back the canonical complex structure on $\mathbb{R}^{2k}$, any framed manifold becomes a $U$-manifold. 

Again, framings admit negatives, are compatible w.r.t.~taking boundaries, and we deem two framed manifolds to be equivalent, $X_1 \sim_{fr} X_2$, if there is a framed manifold $W$ such that $\partial W\cong  X_1 \sqcup -X_2$; the resulting graded ring, the {\em framed bordism ring}, will be denoted by $\Omega^{fr}_*$.

Let us sketch the relation to homotopy theory: We may embed the closed framed manifold $X$ into $\mathbb{R}^N$, for $N$ sufficiently large; up to homotopy, the framing is equivalent to a trivialization of the stable normal bundle $\nu X$, i.e.~we obtain a map
\begin{equation}\label{normal}
\varphi:\nu X\rightarrow \mathbb{R}^{N-n}.
\end{equation}
The normal bundle may be identified with a tubular neighborhood of $X$, and the map \eqref{normal} extends to a map of spheres,
$$\bar{\varphi}:S^N\rightarrow S^{N-n},$$
by sending the complement of the tubular neighborhood to the point at infinity; this is the so-called {\em Pontrjagin-Thom construction}. By Freudenthal's theorem, $\pi_{n+k}S^n$ is independent of $n$ provided that $n>k+1$; we call it the {\em $k^{th}$ stable homotopy group of the sphere} (or the {\em $k^{th}$ stable stem}) and denote it by $\pi^{st}_k$. The Pontrjagin-Thom construction depends only on the framed bordism class of $X$, and it is well-known that the map
$$\Omega^{fr}_k\rightarrow\pi^{st}_k,\quad [X]\mapsto[\bar{\varphi}],$$
is an isomorphism.

\subsection{$\langle k\rangle$-manifolds}
Later on, we want to allow manifolds to have {\em corners}: Recall from \cite{jaenich68} that we can  define a smooth $n$-dimensional manifold $Z$  with corners as being differentiably modeled on the open sets of $\{x\in\mathbb{R}^n|x_1\geq0,\dots,x_n\geq0\}$. If $x\in Z$ is represented by $(x_1,\dots x_n)$ in a local coordinate system, we denote by $c(x)$ the number of zeros in this $n$-tuple; this number is independent of the choice of coordinate system.
Note that $x$ belongs to the closure of at most $c(x)$ different connectedness components of $\{p\in Z|c(p)=1\}$. We call $Z$ a {\em manifold with faces}, if each $x\in Z$ does belong to the closure of $c(x)$ different components of $\{p\in Z|c(p)=1\}$. For a manifold with faces, the closure of a connectedness component of $\{p\in Z|c(p)=1\}$ has the structure of an $(n-1)$-dimensional manifold with corners and is called a {\em connected face} of $Z$; any union of pairwise disjoint connected faces is called a {\em face} of $Z$.

A {\em$\langle k\rangle$-manifold} is an $n$-dimensional manifold with faces $Z$ together with a $k$-tuple $(\partial_1Z,\dots,\partial_kZ)$ of faces such that
\begin{itemize}
\item[(i)]$\partial_1Z\cup\dots\cup\partial_kZ=\partial Z$ and
\item[(ii)]$\partial_iZ\cap\partial_jZ \mbox{ is a face of } \partial_iZ \mbox{ and of } \partial_jZ \mbox{ for } i\neq j.$
\end{itemize}
In particular, a $\langle0\rangle$-manifold is a manifold without boundary; in the situation of a $\langle1\rangle$-manifold, we recover the usual concept of a manifold with boundary.

\subsection{Connections, curvature, and Chern forms}

For the most part, we are adopting the notational conventions of \cite{BGV}: 
Let $E$ be a complex  vector bundle over a compact manifold $X$. We denote the space of all (smooth) sections  by $\Gamma(X,E)$, or, if confusion is unlikely, simply by $\Gamma(E)$. A {\em connection} (or {\em covariant derivative}) on $E$  is a  differential operator 
$$\nabla:\Gamma(E)\rightarrow\Gamma(T^*X\otimes E)$$
satisfying the Leibniz rule; 
usually we restrict our attention to covariant derivatives preserving a given hermitian metric on $E$, in which case we call the connection {\em unitary}. A covariant derivative extends to {\em $E$-valued differential forms}, which we denote by $\Omega^i(X,E)$. The {\em curvature} of a covariant derivative is the $End(E)$-valued two-form on $X$ given by
 $$F(u,v)=\nabla_u\nabla_v-\nabla_v\nabla_u-\nabla_{[u,v]},$$
 where $u$ and $v$ are vector fields on $X$.
Of course, these defintions carry over to the case of real vector bundles as well; in the situation of the Levi-Civita connection, we will denote the $\mathfrak{so}(TX)$-valued {\em Riemannian curvature} form  by $R$.

These concepts may be generalized to  $\mathbb{Z}/2$-graded bundles $E$, which will  be referred to as  {\em superbundles}; for such a superbundle, there is a total $\mathbb{Z}/2$-grading on the space of $E$-valued forms,  
$$\Omega^{\pm}=\sum\Omega^{2i}(X,E^{\pm})\oplus\sum\Omega^{2i+1}(X,E^{\mp}).$$
A {\em superconnection} on $E$ is an odd-parity first-order differential operator $\mathbb{A}:\Omega^{\pm}\rightarrow\Omega^{\mp}$ satisfying the Leibniz rule in the $\mathbb{Z}/2$-graded sense, and its action extends naturally to  $\Omega(X,End(E))$. The {\em curvature} of a superconnection $\mathbb{A}$ is the  operator $\mathbb{A}^2$, which is given by the action of an $End(E)$-valued differential form $F$ of even total degree; it satisfies the {\em Bianchi identity}, $$\mathbb{A}F=0.$$

Due to the supercommutativity of $\Lambda T^*X$, we obtain a canonical bundle map $Str:\Lambda T^*X\otimes End(E)\rightarrow\Lambda T^*X$; applied to sections, this yields the so-called {\em supertrace} on $End(E)$-valued forms. Now, if $\mathbb{A}^2\in\Omega^+(X,End(E))$ is the curvature of a superconnection, we may apply the supertrace to any analytical function $f$ of $\mathbb{A}^2$ to obtain an even  differential form on $X$, the {\em Chern-Weil form} of $\mathbb{A}$ corresponding to $f$; it is closed due to the Bianchi identity.
Furthermore, the {\em transgression formula},
$$\frac{d}{dt}Str(f(\mathbb{A}_t^2))=d\,Str\left(\frac{d\mathbb{A}_t}{dt}f^{\prime}(\mathbb{A}_t^2)\right),$$
applied to the family $\mathbb{A}_t=(1-t)\mathbb{A}_0+t\mathbb{A}_1=\mathbb{A}_0+t{\omega}$ and integrated w.r.t.~$t$, yields the relation
\begin{equation}\label{transgress}
Str(f(\mathbb{A}_1^2))-Str(f(\mathbb{A}_0^2))=d\int_0^1Str(\omega f^{\prime}(\mathbb{A}_t^2))dt.
\end{equation}
In particular, we may apply these constructions to an ordinary bundle with connection, and, contrary to \cite{BGV}, we normalize the characteristic forms such that they represent rational characteristic classes. Thus, if $E$ is a hermitian vector bundle with unitary connection $\nabla^E$ and curvature $F^E$, we denote the {\em Chern character form} by
$$ch\left(\nabla^E\right)=\mbox{tr}\exp\left(\frac{iF^E}{2\pi}\right),$$
and the {\em Todd genus form} by
$$Td\left(\nabla^E\right)=\det\left(\frac{iF/2\pi}{1-\exp\left(-iF/2\pi\right)}\right)=\exp\mbox{tr}\ln\left(\frac{iF/2\pi}{1-\exp\left(-iF/2\pi\right)}\right).$$
Similarly, for a real vector bundle with connection $\nabla$ and curvature $R$ we denote the {\em $\hat{A}$ genus form} by
$$\hat{A}(\nabla)={\det}^{1/2}\left(\frac{R/4\pi i}{\sinh(R/4\pi i)}\right)=\exp{\textstyle\frac{1}{2}}\mbox{tr}\ln\left(\frac{R/4\pi i}{\sinh(R/4\pi i)}\right).$$
Obviously, a unitary connection $\nabla^E$ on a hermitian vector bundle $E$  induces a connection $\nabla^{E_{\mathbb{R}}}$ on the underlying real bundle and, via the trace, a connection $\nabla^{\det E}$ on the hermitian line $\det E=\Lambda^{max}E$, which enables us to express
$$Td\left(\nabla^E\right)=\hat{A}\left(\nabla^{E_{\mathbb{R}}}\right)\exp\mbox{tr}\left(\frac{iF^E}{4\pi}\right).$$
Later on, it will be convenient to be able to change the connection on the underlying real bundle; by means of the transgression formula \eqref{transgress}, we can construct a differential form $cs\left(Td,\nabla_0,\nabla_1\right)$  satisfying
\begin{equation}\label{chernsimons}
d\ cs\left(Td,\nabla_0,\nabla_1\right)=\left\{\hat{A}(\nabla_1^{\mathbb{R}})-\hat{A}(\nabla_0^{\mathbb{R}})\right\}\exp\mbox{tr}\left(\frac{iF^E}{4\pi}\right),
\end{equation}
which we call the {\em Chern-Simons form} (associated to the Todd genus). This construction extends to twisted versions of the Todd genus as well, and we will often denote the corresponding forms simply by $cs$.

\subsection{Some classical index theory}\label{classicalindex}

Recall that the Atiyah-Singer index theorem \cite{Atiyah-Singer-I} identifies the analytical index of an elliptic (i.e.~Fredholm) operator $D$ on a closed oriented manifold $X$ with the topological index; more precisely, we obtain a map that sends the class of the symbol of $D$ - viewed as an element of the compactly supported $K$-theory of the tangent bundle - to the formal difference of the kernel and cokernel of $D$ (an element in the $K$-theory of the point), and the numerical value of the index can be computed using a simple formula  by passing to ordinary cohomology \cite{Atiyah-Singer-III}.

An alternative route to obtaining the index formula makes use of the heat kernel approach (our main reference is \cite{BGV}); its main advantage is the fact that it is {\em local} in the sense that it yields a differential form which is constructed canonically  from the metrical connections on the bundles involved and equates to the index upon integration. For our purposes, it will suffice to restrict our attention to the index theory of twisted $Spin^{\mathbb{C}}$ Dirac operators, although the theorems stated below hold for more general Clifford modules.

We define the group $Spin^{\mathbb{C}}(n)$ as (see e.g.~\cite{lawson-spin})
$$Spin^{\mathbb{C}}(n)=Spin(n)\times_{\mathbb{Z}/2}U(1),$$
and a {\em $Spin^{\mathbb{C}}$ structure} on an oriented Riemannian manifold $X$ of dimension $n$ is a choice of a hermitian line $L$ such that $$w_2(TX)\equiv c_1(L)\mod 2;$$
clearly, the existence of a $Spin^{\mathbb{C}}$ structure is equivalent to the condition $\beta w_2(TX)=0$, where $\beta$ is the Bockstein.

A unitary connection $\nabla^L$ on $L$, together with  the Levi-Civita connection $\nabla^{LC}$ on $TX$, determines a {\em $Spin^{\mathbb{C}}$ connection} $\nabla$; conversely, a $Spin^{\mathbb{C}}$ connections projects (under the canonical twofold covering) to a connection on the $SO(n)\times U(1)$ principal bundle and its associated vector bundles.

Take note that, for a hermitian line $L$ with curvature $F^L$, the  characteristic form given by
$$ch\left(\nabla^{L^{1/2}}\right)=\exp\left(\frac{iF^L}{4\pi}\right)$$
is well-defined even if  $L$ does not admit a global square root.

Obviously, the complex representations of $Spin^{\mathbb{C}}(n)$ are the same as those of $Spin(n)$, so we get an associated {\em complex spinor bundle $S$} over $X$; furthermore, if $n$ is even, $S$ is $\mathbb{Z}/2$-graded by means of the {\em chirality operator} $\gamma$, which is defined by Clifford multiplication with  $i^{n/2}e_1\dots e_n$.

We define the {\em $Spin^{\mathbb{C}}$ Dirac operator} by composing covariant differentiation with Clifford multiplication:
$$\eth={\textstyle\sum}\ e_i\cdot\nabla_{e_i}:\Gamma(X,S)\rightarrow\Gamma(X,S),$$
where the $e_i$ constitute a local orthonormal frame for $TX$.

Now let $\dim X=n$ be even; then $\eth$ anticommutes with $\gamma$, so we may decompose it into
$$\eth=\left(\begin{array}{cc} 0 & \eth^- \\ \eth^+ & 0 \end{array}\right),$$
where $\eth^{\pm}$ are the restrictions to sections of $S^{\pm}$, and $\eth^-$ is the adjoint of $\eth^+$, due to the unitarity of $\nabla$. We define the {\em index} of the $Spin^{\mathbb{C}}$ Dirac operator to be
$$Ind(\eth)=\dim\ker\eth^+-\dim\ker\eth^-.$$

Given a hermitian vector bundle $E$ over $X$ with unitary connection $\nabla^E$, we may form the {\em twisted $Spin^{\mathbb{C}}$ Dirac operator},
$$\eth\otimes E:\Gamma(X,S\otimes E)\rightarrow\Gamma(X,S\otimes E),$$
by using the  tensor product connection on $S\otimes E$, and all of our discussion above holds verbatim. Finally, we can state

\begin{thm}[Atiyah-Singer]\label{AS}
Let $X$ be a closed, oriented, even-dimensional Riemannian  manifold with $Spin^{\mathbb{C}}$ structure defined by a hermitian line $L$ and let $E$ be a hermitian vector bundle with unitary connection. Then the index of the twisted $Spin^{\mathbb{C}}$ Dirac operator  is given by the  formula
$$Ind(\eth\otimes E)=\int_X{\hat{A}(\nabla^{TX,LC})ch(\nabla^{L^{1/2}})ch(\nabla^E)}.$$
\end{thm}

\begin{rmk}
If $X$ is a $U$-manifold, we get a canonical $Spin^{\mathbb{C}}$ structure by setting $L=\Lambda^{max}TX^{st}$. Take note that in this  situation
\begin{equation}\label{canon}
\hat{A}(\nabla^{TX,LC})ch(\nabla^{L^{1/2}})
\end{equation}
represents the Todd class, but we warn the reader that, unless $X$ is K\"ahler, the Levi-Civita connection will {\em not} be compatible with a complex structure on $TX$. On the other hand, if  $X$ is spin, i.e.~$w_2(TX)=0$, we may choose $L$ to be trivial, in which case we essentially recover the situation of `the' Dirac operator, usually denoted $\ /\!\!\!\!D$ (the precise definition of $\ /\!\!\!\!D$, which would take into account the choice of $Spin$ structure and avoid additional complexification, shall not be needed in this thesis).
\end{rmk}

Imposing {\em global} boundary conditions and requiring {\em product type structures} near the boundary, Atiyah, Patodi, and Singer were able to generalize the formula in Theorem \ref{AS} to the situation where $X$ has smooth non-empty boundary \cite{APS1}: 
Restriction to the boundary $\partial X$ induces an  operator 
$$\eth_{\partial X}\otimes E:\Gamma\left((S^+\otimes E)|_{\partial X}\right)\rightarrow\Gamma\left((S^+\otimes E)|_{\partial X}\right)$$
which is formally self-adjoint and elliptic; we may decompose the $L^2$ completion of $\Gamma((S^+\otimes E)|_{\partial X})$ into eigenspaces, and, letting  $P_{\geq0}$ denote the orthogonal projection onto the non-negative part, we define
$$\Gamma(S^+\otimes E,P_{\geq0})=\left\{s\in\Gamma(S^+\otimes E)\ |\ P_{\geq0}\left(s|_{\partial X}\right)=0\right\}.$$
Then, according to \cite{APS1}, 
\begin{equation}\label{apsproblem}
\eth^+\otimes E:\Gamma(S^+\otimes E,P_{\geq0})\rightarrow\Gamma(S^-\otimes E)
\end{equation}
defines an elliptic problem with finite index. Defining the function
$$\eta(\eth_{\partial X}\otimes E,s)=\sum_{\lambda\in spec\setminus\{0\}}{\lambda|\lambda|^{-s-1}},\ \mbox{Re}(s)>\!>1,$$
which extends meromorphically and is holomorphic at $s=0$, we may state:

\begin{thm}[Atiyah-Patodi-Singer]\label{aps}
Let $E$ be a hermitian vector bundle with unitary connection over an  oriented, even-dimensional, compact Riemannian  manifold $X$ with $Spin^{\mathbb{C}}$ structure defined by a hermitian line $L$. Assuming product type structures near the boundary $\partial X$, the index of the twisted $Spin^{\mathbb{C}}$ Dirac operator w.r.t.~the condition \eqref{apsproblem} is given by the  formula
$$Ind_{APS}(\eth  \otimes E)=\int_{X}{\hat{A}(\nabla^{TX,LC})ch(\nabla^{L^{1/2}})ch(\nabla^E)}-\xi(\eth_{\partial X} \otimes E),$$
where
$$\xi(\eth_{\partial X} \otimes E)=\textstyle{\frac{1}{2}}{\eta(\eth_{\partial X}\otimes E,0)}+\textstyle{\frac{1}{2}}\dim\ker(\eth_{\partial X}\otimes E).$$
\end{thm}

Next, we are interested in  generalizations of the results above to the case of {\em families}: Let $\pi:Z\rightarrow B$ be a submersion that defines a fiber bundle with typical fiber $X$ and let the vertical tangent bundle $T(Z/B)$ be equipped with a metric $g^{T(Z/B)}$; furthermore, we make a choice of splitting
$$TZ\cong T(Z/B)\oplus\pi^*TB.$$
Using an auxiliary metric on $TB$, we form the metric
\begin{equation}\label{metric}
g=g^{T(Z/B)}\oplus\pi^*g^{TB};
\end{equation}
let $\nabla^{TZ}$ be its Levi-Civita connection, and let $P:TZ\rightarrow T(Z/B)$ denote the orthogonal projection;  by setting
\begin{equation}\label{nablatzb}
\nabla^{T(Z/B)}=P\circ \nabla^{TZ}\circ P,
\end{equation}
we obtain  a  connection $\nabla^{T(Z/B)}$ on $T(Z/B)$; in particular, it does not depend on the metric $g^{TB}$ and restricts to the Levi-Civita connection on each fiber.
Take note that the connection on $TZ$ given by
$$\nabla^{\oplus}=\nabla^{T(Z/B)}\oplus\pi^*\nabla^{TB}$$
preserves the metric \eqref{metric}, but it is not torsion-free. We define the {\em curvature of the fibration}
$$T(u,v)=-P[u,v],$$
where $u$ and $v$ are horizontal vectors.

Fixing a $Spin^{\mathbb{C}}$ structure on the vertical tangent, we may form the associated complex spinor bundle $S$; given a hermitian vector bundle $E$ over $Z$ with unitary connection $\nabla^E$, we get a {\em family of twisted Dirac operators} $\{\eth_b\otimes E\}_{b\in B}$, where
\begin{equation}\label{Diracfamily}
\eth_b\otimes E:\Gamma\left((S\otimes E)|_{\pi^{-1}b}\right)\rightarrow\Gamma\left((S\otimes E)|_{\pi^{-1}b}\right).
\end{equation}
Let us briefly comment on the situation where the fiber $X$ is even-dimensional and {\em closed} \cite{Atiyah-Singer-IV}: If we assume that the kernel and cokernel of $\eth^+\otimes E$ form vector bundles over $B$, we can define the {\em index bundle} as the formal difference class,
\begin{equation}\label{indexbundle}
Ind(\eth\otimes E)=\left[\ker\left(\eth^+\otimes E\right)\ominus\ker\left(\eth^-\otimes E\right)\right]\in K(B).
\end{equation}
The caveat is that, generically, the kernels do {\em not} form vector bundles. However, this situation may be remedied by modifying the  family of  operators using a compact perturbation; as such a perturbation does not affect the index, we obtain a well-defined $K$-theory class, which, by abuse of notation, will still be denoted as in \eqref{indexbundle}. Furthermore, the $K$-theoretical construction of the index map generalizes to the families situation, and, upon applying the Chern character, one obtains the following formula in $H^*(B,\mathbb{Q})$ \cite{Atiyah-Singer-IV}:
\begin{equation}\label{ASfamilyformula}
ch(Ind(\eth \otimes E))=\pi_*[{\hat{A}(T(Z/B))\exp({\textstyle\frac{1}{2}}c_1(L))ch(E)}].
\end{equation}
There is a heat kernel proof of this result, due to Bismut \cite{bismut85}, which we shall omit; however, it introduces a concept that will be needed later on: 
Given a family  of twisted Dirac operators  $D$ constructed in the situation of a closed (but not necessarily even-dimensional) typical fiber, we define the {\em Bismut superconnection} to be 
\begin{equation}\label{bconnection}
\mathbb{A}_t=\sqrt{t}D+\tilde{\nabla}-\frac{c(T)}{4\sqrt{t}},
\end{equation}
where 
 $\tilde{\nabla}$ is the natural lift of $\nabla$ to the infinite dimensional bundle of sections of the vertical spinor bundle and $c(T)$ denotes Clifford multiplication with the curvature of the fibration; we refer the reader to \cite{bismut85}, \cite{BGV} for further details.
 
With this in mind, let us consider the family \eqref{Diracfamily} in the situation of an even-dimensional typical fiber $X$ with smooth, non-empty  boundary, where $g^{T(Z/B)}$ is assumed to be of product type near the boundary. Then the induced family of Dirac operators on the boundary gives rise to an associated Bismut superconnection  $\mathbb{A}_t$, and, if the kernel of $D$ forms a vector bundle over $B$, then the following differential form, 
\begin{equation*}\label{etaunnorm}
\tilde{\eta}=\frac{1}{\sqrt{\pi}}\int_0^{\infty}{\text{Tr}}^{ev}\left[\frac{d\mathbb{A}_t}{dt}\exp(-\mathbb{A}_t^2)\right]dt,
\end{equation*}
where $\text{Tr}^{ev}$ denotes the even form part of the trace, is well-defined, see e.g.~\cite{BGV}. This defines an even form on $B$, and its rescaled version
 $$\hat{\eta}|_{\deg2k}=(2\pi i)^{-k}\tilde{\eta}|_{\deg2k},$$
will be  referred to  as the {\em $\hat{\eta}$-form}.

Finally, a generalization of Theorem \ref{aps} to the situation of families, which is due to Bismut and Cheeger \cite{BC1}, \cite{BC2}, \cite{BCrem}, reads:
\begin{thm}[Bismut-Cheeger]\label{bc} Let $Z\rightarrow B$ be a fiber bundle, the even-dimensional typical fiber $X$ having smooth, non-empty boundary, and let $\eth\otimes E$ be a twisted Dirac family constructed using a metric that is of product type near the  boundary. 
If the kernel of the twisted Dirac operator induced on the fiberwise boundary is of constant rank, then the index bundle w.r.t.~the APS boundary condition is well-defined; furthermore, a representative in cohomology of the Chern character of the index bundle is given by the following smooth differential form on $B$:
$$\int_{Z/B}\left\{\hat{A}(\nabla^{T(Z/B)})ch(\nabla^{L^{1/2}})ch(\nabla^E)\right\}-\hat{\eta}(\eth_{\partial} \otimes E)-{\textstyle{\frac{1}{2}}}ch\left(\nabla^{\ker\left(\eth_{\partial}\otimes E\right)}\right).$$
\end{thm}

\subsection{Hirzebruch elliptic genera}
Let us recall the definition of the  Hirzebruch elliptic genus of level $N$ associated to the congruence subgroup $\Gamma=\Gamma_1(N)$ \cite{HBJ}: For fixed $N>1$, let  $\zeta$ be  a primitive $N^{th}$ root of unity and let $q=\exp(2\pi i\tau)$, $\tau\in\mathfrak{h}$. We consider the function 
$$Ell^{\Gamma}(x)=\frac{x}{1-e^{-x}}\frac{1-\zeta e^{-x}}{1-\zeta}\prod_{n=1}^{\infty}\frac{(1-q^n)^2}{(1-q^ne^x)(1-q^ne^{-x})}\frac{1-q^ne^x/\zeta}{1-q^n/\zeta}\frac{1-\zeta q^ne^{-x}}{1-\zeta q^n}$$
as a power series in the indeterminate $x$ of degree two. Making use of standard results on elliptic functions, it can be shown that the coefficient of $x^m$ is a modular form of weight $m$ w.r.t.~$\Gamma$ (confer e.g.~\cite{HBJ}); an alternative route is to rewrite the power series such that modularity becomes manifest, see appendix \ref{expell}.

For a hermitian vector bundle $E$ with unitary connection $\nabla^E$, we define the {\em elliptic genus form} to be
\begin{equation}\label{ellform}
Ell^{\Gamma}\left(\nabla^E\right)=\exp\mbox{tr}\left(\ln Ell^{\Gamma}\left(\frac{iF^E}{2\pi}\right)\right).
\end{equation}
If we introduce the power operations with respect to a formal parameter $t$,
$$S_t(V)=\bigoplus_{k\geq0}t^kS^k(V),\  \ \Lambda_t(V)=\bigoplus_{k\geq0}t^k\Lambda^k(V),$$
which extend to $K$-theory classes in the obvious way \cite{atiyahktheory}, we see that the underlying cohomology class of \eqref{ellform} is given by
\begin{equation}\label{elli}
Ell^{\Gamma}(E)=\overline{\chi}_{-\zeta}(E)ch\left(\bigotimes_{n=1}^{\infty} S_{q^n}(\overline{E\otimes\mathbb{C}})\otimes\Lambda_{-q^n/\zeta}\overline{E}\otimes\Lambda_{-\zeta q^n}\overline{E^*}\right),
\end{equation}
where the bar denotes virtual reduction of the complex bundles involved, and $\overline{\chi}_y$ is the {\em{stable}} $\chi_y$-genus, i.e.
$$\overline{\chi}_y(E)={(1+y)^{-rkE}}\chi_y(E)={(1+y)^{-rkE}}Td(E)ch(\Lambda_yE^*).$$
We have the following well-known result:
\begin{prop}\label{intelli}
Let $X$ be a closed $U$-manifold. Then the elliptic genus of $X$ has an integral $q$-expansion, i.e.
$$\langle Ell^{\Gamma}(TX),[X]\rangle\in\mathbb{Z}[\zeta,1/N][\![q]\!].$$
\end{prop}
\begin{proof} Multiplying the LHS by $(1-\zeta)^{rk(TX^{st})}$, expanding the formal power series, and grouping powers of $\zeta$, we see that, by Theorem \ref{AS}, every coefficient admits an interpretation as the index of a suitably twisted $Spin^{\mathbb{C}}$ Dirac operator. Since $(1-\zeta)^{-1}\in  \mathbb{Z}[\zeta,1/N]$, the claim is proven.
\end{proof}
For convenience, we are going to delete the fixed group $\Gamma$ from the notation (after all, we suppressed the dependence on $\zeta$ from the very beginning); furthermore, we introduce the abbreviations 
\begin{equation*}
Ell_0=Ell|_{q=0},\quad \widetilde{Ell}=Ell-Ell_0.
\end{equation*}

\subsection{The $e$-invariant}

The original formulation of the $e$-invariant, $e:\pi^{st}_{2k+1}\rightarrow \mathbb{Q/Z}$, is due to Adams \cite{j4}; for our purposes however, it will be more convenient to use the cobordism description given by  Conner and Floyd \cite{cf66}:
\begin{df}\label{ufr}
A $(U,fr)$-manifold is a compact $U$-manifold $X$ with smooth boundary and a trivialization of $E\cong TX^{st}$ over the boundary, i.e.~a bundle map
$$\psi:E|_{\partial X}\cong\partial X\times\mathbb{C}^k.$$
\end{df}
In particular, $\psi$ provides a framing for $\partial X$; using the relative characteristic classes of the complex vector bundle $E\cong TX^{st}$,  the {\em complex $e$-invariant} of the framed bordism class of $\partial X$  is defined to be
\begin{equation}\label{e_c}
e_{\mathbb{C}}(\partial X)\equiv\langle{Td}(E),[X,\partial X]\rangle\mod\mathbb{Z}.
\end{equation}
As mentioned in the introduction,  the $e$-invariant admits an interpretation in terms of index theory, and  it can be related to (and computed from) the spectral asymmetry encoded into $\eta(\eth)$ \cite{APS2}; we begin by rephrasing the RHS of \eqref{e_c}: The framing induces a hermitian metric on $E|_{\partial X}$; we extend it to $E$ such that it is of product type near the boundary. Furthermore, we endow $E$ with a unitary connection $\nabla^E$ that restricts to the canonical flat connection specified by the trivialization, i.e.~the one w.r.t.~which the frame is parallel; this enables us to rewrite $\langle{Td}(E),[X,\partial X]\rangle=\int_X{Td(\nabla^E)}$. 

Now we can show that $e_{\mathbb{C}}$ is well-defined: Let $X'$ be another $(U,fr)$-manifold having the same framed boundary, i.e.\ $\partial X'\cong\partial X$, and let $W$ be the closed $U$-manifold obtained by gluing $X$ and $-X'$ along the boundary. Then, by the linearity of the integral, the relative Todd genera of $X$ and $X'$ differ by the Todd genus of $W$, which is an integer by Theorem \ref{AS}.  Noting that the integrand is trivial on any framed bordism $W'$ shows that $e_{\mathbb{C}}$ depends only on the framed bordism class of $\partial X$.
 
In order to compute the $e_{\mathbb{C}}$-invariant analytically, we consider the canonical $Spin^{\mathbb{C}}$ Dirac operator on the $(U,fr)$-manifold $X$ and apply Theorem \ref{aps}, 
$$\int_X{Td\left(\nabla^{TX,LC}\right)}\equiv \xi\left(\eth_{\partial X}\right)\mod\mathbb{Z},$$
where we used $Td(\nabla^{TX,LC})$ as the shorthand notation for \eqref{canon}, i.e.~the local index form associated to the $Spin^{\mathbb C}$ Dirac operator built from the Levi-Civita connection $\nabla^{LC}$ on $TX$ and the connection on $\det E$ (induced by $\nabla^E$). This implies
$$e_{\mathbb{C}}(\partial X)\equiv\xi\left(\eth_{\partial X}\right)+\int_X{\left\{Td\left(\nabla^E\right)-Td\left(\nabla^{TX,LC}\right)\right\}}\mod\mathbb{Z},$$
but with the help of \eqref{chernsimons} and Stokes' theorem, the integral can be reduced to an integral  over $\partial X$. Thus, the $e$-invariant is  computable from geometrical data on $M=\partial X$ itself:
\begin{equation}\label{e_c_2}
e_{\mathbb{C}}\left(M\right)\equiv\xi\left(\eth_M\right)+\int_Mcs\mod\mathbb{Z}.
\end{equation}

\begin{rmk}\label{r_vs_c}
In \cite{APS2}, Atiyah, Patodi, and Singer actually treat the {\em real} $e$-invariant, $e_{\mathbb{R}}:\pi^{st}_{4k-1}\rightarrow \mathbb{Q/Z}$: Since $MSpin_{4k-1}=0$, the framed manifold $M$ is the boundary of  a $Spin$-manifold $N$. Considering the Dirac operator $/\!\!\!\!D$,  the quaternionic structure of the spinors in dimensions $4k,4k-1,$ for $2\nmid k$ implies that  the kernels are even-dimensional, so one obtains the refined result
$$\epsilon(k)\left\{ \int_M{cs}+\xi(/\!\!\!\!D_{M})\right\}\equiv\epsilon(k)\langle\hat{{A}}(TN),[N,M]\rangle \equiv e_{\mathbb R}(M) \mod\mathbb{Z},$$
where $\epsilon(k)=1$ if $2|k$ and $1/2$ otherwise. 
We would like to point out that we have $MSU_{4k-1}=0$ as well \cite{cf66}; furthermore,  the first Chern class of an $SU$-manifold is trivial, in which case the Todd genus coincides with $\hat A$, showing  that $e_{\mathbb{R}}/\epsilon\equiv e_{\mathbb{C}}\mod\mathbb{Z}$.
\end{rmk}

Admittedly, the formula \eqref{e_c_2} seems of little practical use, as one rarely is in the situation to compute the spectrum of $\eth$ explicitly. There are, however, some notable exceptions: In particular, the analytical computation of the (real) $e$-invariant for nilmanifolds covered by Heisenberg groups has been carried out by Deninger and Singhof, thus exhibiting a family  representing (twice) the generator of $Im(J)$ in dimension $8k+3$ ($8k+7$) \cite{DSheisenberg}. On the other hand, index theory considerations yield a vanishing theorem for compact Lie groups of higher rank \cite{atiyah-smith}; strictly speaking, this result is formulated for $e_{\mathbb{R}}$ and holds under the additional assumption of semi-simplicity. The vanishing of the  complex $e$-invariant for higher rank Lie groups (not necessarily semi-simple) can also be deduced from the algebraic-topological results of \cite{kna78}, see \cite{Lau00} for a geometrical interpretation.

\subsection{The topological $f$-invariant}

Recall that the Adams-Novikov spectral sequence gives rise to a filtration of the stable stems. Geometrically, the AN filtration can be understood in terms of manifolds with corners \cite{Lau00}: A framed manifold is {\em in $k^{th}$ filtration} if it occurs as the codimension-$k$ corner of a so-called $(U,fr)^k$-manifold; in particular, we already defined the $e_{\mathbb{C}}$-invariant for boundaries of $(U,fr)$-manifolds, i.e.~for manifolds in first filtration.
\begin{df}\label{ufr2}
A {{$(U,fr)^2$-manifold}} is a compact $\langle2\rangle$-manifold $Z$  together with two complex vector bundles $E_1$, $E_2$, with trivializations  over the faces $\partial_1 Z$, $\partial_2 Z$, respectively, i.e.~a choice of bundle maps $\psi_i:E_i|_{\partial_iZ}\cong \partial_iZ \times \mathbb{C}^{k_i}$, and an isomorphism $TZ^{st} \cong E_1 \oplus E_2$ (in the stable sense).
\end{df}

Fixing $\Gamma=\Gamma_1(N)$, we set $\mathbb{Z}^{\Gamma}=\mathbb{Z}[\zeta,1/N]$ and denote by $M^{\Gamma}_*$ the graded ring of modular forms w.r.t.~$\Gamma$ which expand integrally, i.e.~which lie in $\mathbb{Z}^{\Gamma}[\![q]\!]$. We define the ring of {\em divided congruences} $D^{\Gamma}$ to consist of those rational combinations of modular forms which expand integrally; this ring can be filtered by setting
$$D_k^{\Gamma}=\left\{\left.f={\sum_{i=0}^{k}}f_i\ \right| f_i\in M_i^{\Gamma}\otimes\mathbb{Q},\ f\in\mathbb{Z}^{\Gamma}[\![q]\!]\right\}.$$
Finally, we introduce $\underline{\underline{D}}^{\Gamma}_{k}=D^{\Gamma}_k+M_0^{\Gamma}\otimes\mathbb{Q}+M_k^{\Gamma}\otimes\mathbb{Q}$.
 
Let $M^{2n}$ be the codimension-two corner of a $(U,fr)^2$-manifold $Z$. Using the relative Chern classes of the split tangent bundle,  the {\em$f$-invariant} of the framed bordism class of $M$ is defined to be
\begin{equation}\label{f}
f(M)\equiv\langle(Ell(E_1)-1)(Ell_{0}(E_2)-1),[Z,\partial Z]\rangle\mod\underline{\underline{D}}^{\Gamma}_{n+1},
\end{equation}
hence $f$ takes values in $\underline{\underline{D}}^{\Gamma}_{n+1}\otimes\mathbb{Q/Z}$. 
We refer  to  \cite{Lau99} and \cite{Lau00} for details concerning the homotopy theoretical construction  and its interpretation as  a genus arising from an ($MU^{\langle2\rangle}$-) orientation of a suitable $\langle2\rangle$-spectrum $E$; what we are aiming  at in this thesis, however, is to provide a geometrical  interpretation of  \eqref{f}, in  a fashion similar to the last section.

\section{The geometrical $\check{f}$-invariant}\label{main}

\subsection{Divided congruences from trivialized vector bundles}

Recall from \cite{APS2} that the spectral information encoded in $\xi$ can be used to formulate an invariant for {\em{flat}} vector bundles: Let $M$ be a closed $U$-manifold of odd dimension and let $E$ be a  hermitian vector bundle with flat unitary connection $\nabla^E$; then the expression
\begin{equation}\label{ro}
\tilde{\xi}\left(\nabla^E\right) \equiv \xi(\eth \otimes E) - rk(E)\xi({\eth})\mod\mathbb{Z},
\end{equation}
where $\eth$ denotes the canonical $Spin^{\mathbb{C}}$ Dirac operator on $M$, is independent of the metric, hence an $\mathbb{R/Z}$-valued invariant of the flat bundle $E$. This can be seen as follows: Since $\Omega_{odd}^U(BU)=0$, we can find a $U$-manifold $X$ with boundary $M$ and a vector bundle $\hat{E}$ over $X$ (not necessarily flat) that extends $E$; then Theorem \ref{aps} yields
\begin{equation}\label{rorephrased}
\tilde{\xi}\left(\nabla^E\right) \equiv \int_X \left\{ ch\left(\nabla^{\hat{E}}\right) - rk \hat{E} \right\} Td(\nabla^{TX})\mod\mathbb{Z},
\end{equation}
which is the mod $\mathbb{Z}$ reduction of the evaluation of a {\em{real}} relative cohomology class and  easily seen to be independent of all choices. Obviously, this invariance property persists if we couple $E$ to twisted versions of $\eth$, 
i.e.~we may consider
$$\tilde{\xi}\left(\nabla^E\right)\otimes F\equiv \xi(\eth \otimes F\otimes E) - rk(E)\xi({\eth \otimes F})\mod\mathbb{Z},$$
for some hermitian vector bundle $F$ with unitary connection. Take note that even if $E$ is trivial, the invariant can be non-zero; however, a choice of trivialization induces a canonical flat connection  $\nabla^{p.g.}$, namely the one with respect to which the global section trivializing the principal $U$-bundle is parallel (strictly speaking, this yields is a connection on the principal bundle, but since it canonically induces connections on any associated vector bundle, we do not bother to distinguish). In physics terminology, this  is a so-called {\em{globally pure gauge}} connection, and it has vanishing holonomy along all closed paths in $M$. 

\begin{lem}\label{keylemma}
Let $E$ be a trivialized hermitian vector bundle over an odd-dimensional closed $U$-manifold $M$, and let $\nabla^{p.g.}$ be the unitary connection preserving the trivialization. Then we have
$$\tilde{\xi}\left(\nabla^{p.g.}\right) \equiv 0\mod\mathbb{Z};$$
this remains true if we twist the $Spin^{\mathbb{C}}$ Dirac operator with an  auxiliary hermitian vector bundle $F$ with unitary connection.
\end{lem}
\begin{proof} Since we do not need the full generality of   \cite[Theorem 3.3]{APS2}, we may argue as follows: Clearly, we have
\begin{equation*}\label{minus}
-\xi(\eth_M) \equiv \xi(\eth_{-M})\mod\mathbb{Z},
\end{equation*}
so we can interpret each summand of \eqref{ro} separately in terms of index theory on manifolds $X$, $X'$ with opposite boundary; in particular, we may represent \eqref{ro} by
$$\int_XTd(\nabla^{TX})ch(\nabla^{\hat{E}})+\int_{X'}Td(\nabla^{TX'})\ rkE.$$
Furthermore, using the trivialization of $E$, $\hat{E}$ and the trivial bundle $rkE$ patch together to form a hermitian bundle over the closed $U$-manifold $X\cup X'$; by Theorem \ref{AS}, the sum of the integrals  yields an integer. For the twisted case, we notice that $F$ also extends to $X$ and $X'$ (since $\Omega_{odd}^U(BU\times BU)=0$) and that the multiplication of the integrands by $ch(\nabla^F)$ does not change the validity of our argument above.
\end{proof}

This result might seem a little bit dull, but it enables us to establish a surprising relation between trivialized vector bundles and divided congruences: 

\begin{thm}\label{div}
Let $X$ be a compact manifold of dimension $2n$, and let $E_1$, $E_2$, and $F$  be hermitian vector bundles over $X$ such that $E_1\oplus E_2\cong TX^{st}$,  and that there is a given trivialization  $\psi:E_1|_{\partial X}\cong\partial X\times \mathbb{C}^k$. 
Equip $E_1$ with any unitary connection $\nabla^{E_1}$ that restricts to the pure gauge connection on the boundary. Then, for  arbitrary unitary connections $\nabla^{E_2}$ and $\nabla^{F}$, 
$$\int_X{\widetilde{Ell}(\nabla^{E_1})Ell_0(\nabla^{E_2})ch(\nabla^{F})\in D_n^{\Gamma}}.$$
\end{thm}

\begin{proof}
By the multiplicativity of $Ell_0$, the integrand may be rewritten as 
$$Ell_0(\nabla_1\oplus\nabla_2)ch\left\{\left(\bigotimes_{n=1}^{\infty} S_{q^n}\overline{E_1\otimes\mathbb{C}}\otimes\Lambda_{-q^n/\zeta}\overline{E_1}\otimes\Lambda_{-\zeta q^n}\overline{E_1^*}\right)-1\right\}ch(\nabla^F),$$
where we dropped the symbol $\nabla$ inside the curly brackets in favor of notational simplicity. Furthermore, $Ell_0$ itself is a twisted version of the Todd genus, so
$$Ell_0(\nabla_1\oplus\nabla_2)=(1-\zeta)^{-rk(E_1\oplus E_2)}Td(\nabla_1\oplus\nabla_2)ch\left(\nabla^{\Lambda_{-\zeta}(E_1\oplus E_2)^*}\right),$$
but $Td(\nabla_1\oplus\nabla_2)$ agrees with $Td(\nabla^{TX})$ up to the differential of a Chern-Simons term; the latter does not contribute to the integral, since it gets multiplied by a relative characteristic form vanishing on $\partial X$. Then we observe that the formal vector bundle inside the curly brackets expands such that each summand contains at least one factor that is a virtually reduced bundle, so we may apply Theorem \ref{aps} and Lemma \ref{keylemma} to establish integrality. 
On the other hand, the integral takes values in inhomogeneous modular forms, due to the factor $\widetilde{Ell}(\nabla_1)$; in fact, we may decompose the latter into the difference of $Ell(\nabla_1)-1$ and $Ell_0(\nabla_1)-1$, which, due to the pure gauge condition, represent classes in $\bigoplus_k \left(H^{2k}(X,\partial X;\mathbb{Q})\otimes M^{\Gamma}_k\right)$ and $H^{even}(X,\partial X;\mathbb{Q[\zeta]})$, respectively.
\end{proof}

\begin{ex}\label{divex}
We may consider the following special case: If $X$ is a $(U,fr)$-manifold of dimension $2n$,  we may choose $E_1\cong TX^{st}$, $E_2=0$ and $F=1$. Then we see that the  a priori {\em rational} modular form of weight $n$ given by the  elliptic genus of $X$ actually admits an {\em integral} $q$-expansion once we remove its constant term:
$$\int_XEll(\nabla^{E_1})-\int_XEll_0(\nabla^{E_1})=\int_X\widetilde{Ell}(\nabla^{E_1})\in D_n^{\Gamma}.$$
\end{ex}
\begin{rmk}
The preceding example seems to be well-known; for instance, it may also be deduced from the results of \cite{Lau99}, albeit with considerably more effort, at least compared to the simple {\em geometrical} statement of Theorem \ref{div}.
\end{rmk}

\subsection{Construction of $\check{f}$}\label{towardstheformula}

Having established an integrality result which may serve as a substitute for the Atiyah-Singer index theorem, we may now  leave the realm of manifolds with smooth boundary and turn our attention to manifolds with corners of codimension two. In the following, we consider $(U,fr)^2$-manifolds $Z$  of (positive) even dimension. We also want additional geometrical structures: From now on, we endow all our bundles with hermitian metrics and unitary connections which are of product type near the respective faces; furthermore, we want these connections to preserve the respective trivializations  on the faces, i.e.~we require that they restrict to the pure gauge ones. Let us call these connections {\em compatible}.

\begin{df}
For a $(U,fr)^2$-manifold $Z$ of real dimension $2n+2$ and any compatible connections $\nabla_i$ on the $E_i$, we set
$$\check{F}(Z,\nabla_1,\nabla_2)=\int_Z{\widetilde{Ell}(\nabla_1)Ell_0(\nabla_2)}\in\bigoplus_{k=0}^{n+1}M_k^{\Gamma}\otimes\mathbb{R}.$$
\end{df}

If we allow  $(U,fr)^2$-manifolds with empty corner, we obtain the following integrality results for $\check{F}$:

\begin{prop}\label{interior}
Let $E_1$ and $E_2$ be hermitian vector bundles over a closed manifold $X$ of dimension $2n+2$ such that $E_1\oplus E_2\cong TX^{st}$ \textup{(}in the stable sense\textup{)}. For any unitary connections $\nabla_i$  we have
$$\check{F}(X,\nabla_1,\nabla_2)\equiv0\mod D_{n+1}^{\Gamma}.$$
\end{prop}
\begin{proof} 
Integrality is established by either making use of Theorem \ref{AS}, or by applying Theorem \ref{div} to the situation of $\partial Z=\emptyset$.
\end{proof}

\begin{prop}\label{bdry}
Let $X$ be a compact manifold of dimension $2n+2$ and let $E$ and $F$ be hermitian vector bundles over $X$ such that $E\oplus F\cong TX^{st}$ \textup{(}in the stable sense\textup{)} and that there is a given trivialization  $\psi:F|_{\partial X}\cong\partial X\times \mathbb{C}^k$. 
Regard  $X$ as a $(U,fr)^2$-manifold with empty corner, and choose any unitary connection $\nabla^F$ that restricts to the pure gauge one on the boundary. Then we have:
\begin{itemize}
\item[\textup{(i)}] $\check{F}(X,\nabla^F,\nabla^E)\equiv0\mod D_{n+1}^{\Gamma}.$
\item[\textup{(ii)}] $\check{F}(X,\nabla^E,\nabla^F)\equiv0\mod D_{n+1}^{\Gamma}+M^{\Gamma}_0\otimes\mathbb{R}+M_{n+1}^{\Gamma}\otimes\mathbb{R}.$
\end{itemize}
\end{prop}
\begin{proof}
The first statement is clear by Theorem \ref{div}. For the second statement, we  make use of the identity 
\begin{equation}\label{theidentity}
\widetilde{Ell}(\nabla_1)Ell_0(\nabla_2)=\widetilde{Ell}(\nabla^{\oplus})-\widetilde{Ell}(\nabla_1)\widetilde{Ell}(\nabla_2)-\widetilde{Ell}(\nabla_2)Ell_0(\nabla_1),
\end{equation}
where $\nabla^{\oplus}=\nabla_1\oplus\nabla_2$. Thus, up to modular terms of top and zero weight, i.e.~up to $\int\widetilde{Ell}(\nabla^{\oplus})$, we can express $\check{F}(X,\nabla^E,\nabla^F)$ in terms to which Theorem \ref{div} applies.
\end{proof}

Furthermore, we have a stability result concerning the splitting, namely:

\begin{prop}\label{stable}
Let $(Z,E_1,E_2)$ be a $(U,fr)^2$-manifold of dimension $2n+2$. Suppose that we have a different splitting  of the tangent bundle of the same manifold, $TZ\cong E_1'\oplus E_2'$, such that $E_1|_{\partial Z}=E_1'|_{\partial Z}\oplus F|_{\partial Z}$, where $F$ is trivialized over all of $\partial Z$, and  let $\nabla_0$ be any unitary connection on $F$ that restricts to the pure gauge one on both faces. Then we have
$$\check{F}(Z,\nabla_1'\oplus\nabla_0,\nabla_2)\equiv \check{F}(Z,\nabla_1',\nabla_0\oplus\nabla_2) \mod D_{n+1}^{\Gamma}.$$
\end{prop} 
\begin{proof}
From \eqref{theidentity} we deduce
$$\widetilde{Ell}\left(\nabla_1'\oplus\nabla_0\right)-\widetilde{Ell}\left(\nabla_1'\right)Ell_0\left(\nabla_0\right)=Ell\left(\nabla_1'\right)\widetilde{Ell}\left(\nabla_0\right);$$
then we multiply by $Ell_0(\nabla_2)$, integrate, and apply Theorem \ref{div} making use of the fact that, since the integrand vanishes near the corner, the integral over $Z$ will yield the same result as the integral over a manifold $Z_s$ obtained from smoothing the corner.
\end{proof}

The preceding results suggest the following:

\begin{df}
Let $M$ be a closed manifold of positive even dimension $2n$, which is the corner of a $(U,fr)^2$-manifold $Z$, and therefore inherits a splitting of its framing. Then, using compatible connections, we set
\begin{equation}\label{fcheck}
\check{f}(M,\nabla_1|_M,\nabla_2|_M)\equiv\check{F}(Z,\nabla_1,\nabla_2)\mod D_{n+1}^{\Gamma}+M^{\Gamma}_0\otimes\mathbb{R}+M_{n+1}^{\Gamma}\otimes\mathbb{R}.
\end{equation}
We call $\check{f}$ the {\em{geometrical f-invariant}}. In fact, this is well-defined.
\end{df}
\begin{proof} 
First of all, two $(U,fr)^2$-manifolds $Z_1$, $Z_2$, having  in common one face (and therefore having the same corner), in the sense that, say, $\partial_1Z_1\cong\partial_1Z_2$ together with identifications of the respective $E_i$ thereon, give rise to congruent $\check{f}$-invariants, for we may  glue $-Z_1$ and $Z_2$ along this face. By assumption, the metric and connections near the boundary are of product type, so everything fits  together to yield a manifold $Y$ with smooth boundary to which  we may apply Proposition \ref{bdry}. Similarly, any other manifold $Z_3$ coinciding with $Z_2$ on the {\em{other}} face will have the same $\check{f}$ as well. Finally, given $Z_1$ and $Z_3$ having in common the corner $M$ (together with an identification of the trivialized vector bundles $E_i|$ thereon), there always exists a suitable $(U,fr)^2$-manifold $Z_2$: The faces $\partial_1Z_3$ and $\partial_2Z_1$ fit together along $M$ to form a topological $U$-manifold $N$ of odd dimension; furthermore the restrictions of the vector bundles (and the connections thereon) fit together to form vector bundles $\tilde{E}_i$. Since $\Omega^U_{odd}(BU)=0$, there exists a $U$-manifold $P$ such that $\partial P\cong N$, and a complex vector bundle $F_1$ over $P$ extending $\tilde{E}_1$. Specifying a vector bundle representative of the $K$-theory class $[TP^{st}\ominus F_1]$,  $P$ is turned into the desired $(U,fr)^2$-manifold $Z_2$. 
\end{proof}

Thus, we have succeeded in constructing a geometrical invariant of the corner of a $(U,fr)^2$-manifold; furthermore,  $\check{f}$ bears a striking resemblance to the topological $f$-invariant, and, in fact, the former will serve as a tool for the index theoretical computation of the latter in the following section.

\section{Calculability}\label{calculable?}

It is a natural question to ask whether the (geometrical) $f$-invariant is computable using index theory. In order to establish a formula that is similar to \eqref{e_c_2}, we have to address the following problems:
\begin{itemize}
\item[(i)] Analysis: A good starting point would be an index theorem on manifolds with corners of codimension two - alas, there are no theorems comparable to the generality of \cite{APS1}; however, we would like to mention \cite{mue96}, where an index formula is proved under the assumption that the induced Dirac operators are {\em invertible}. The results of \cite{stmwc} show that, without this assumption, it is still possible to obtain an `index formula', but the latter holds only {\em modulo the integers}.
\item[(ii)] Modularity: We want our formula to yield a result that is still recognizable as a combination of modular forms, but this property would inadvertently be spoiled by  reducing modulo the index (which takes values in $\mathbb{Z}^{\Gamma}[\![q]\!]$); furthermore, working one operator at a time, we obtain just a finite amount of coefficients of a $q$-expansion. Unfortunately, it is unclear under which conditions a finite amount of reduced coefficients can be lifted to an inhomogeneous combination of modular forms, and this task is complicated  by the fact  that $\check{f}$ is defined only up to {\em real} modular forms of top degree.
\item[(iii)] Geometry: Lastly, we have to keep in mind that the definition of the $f$-invariant makes use of a $(U,fr)^2$-manifold, the construction of which is also a non-trivial task. 
\end{itemize}

Our approach is to simplify matters by making some assumptions on the underlying geometry, i.e.\ we seek out $(U,fr)^2$-manifolds that are sufficiently `nice', in the sense that they allow the problems (i) and (ii) to be resolved. 

\subsection{Corners via fiber bundles}\label{basicsetup}

As a first step, we restrict our attention to manifolds of the following form:

\begin{df}\label{two_f}
We define a {\em $\langle 2\rangle_f$-manifold} to be a compact $\langle2\rangle$-manifold that is a fiber bundle
$$\pi:Z\longrightarrow B,$$
where both the fiber $X$ and the base $B$ are even-dimensional compact $\langle1\rangle$-manifolds, and the faces are given by
$$Z'=\partial_1Z,\ \ X\rightarrow Z'\rightarrow \partial B,$$ 
$$Z''=\partial_2Z,\ \ \partial X\rightarrow Z''\rightarrow B,$$
which are fiber bundles themselves.
\end{df}
Expecting such manifolds to be accessible to index theory considerations, we proceed along the lines of section \ref{classicalindex}, i.e.\ we introduce metrics $g^{T(Z/B)}$ and $g^{TB}$  (which are assumed to be of product type near the respective faces), make a choice of splitting $TZ\cong T(Z/B)\oplus\pi^*TB$,  and construct the connection $\nabla^{\oplus}=\nabla^{T(Z/B)}\oplus\pi^*\nabla^{TB}$.

As a model situation, we consider $Z$ to be equipped with fixed $Spin$ structures on the bundles; these induce natural orientations, and we may compute  the integral of the $\hat{A}$ genus form using integration over the fiber:
$$ \int_Z{\hat{A}(\nabla^{\oplus})}=\int_Z{\hat{A}(\pi^*\nabla^{TB})\hat{A}(\nabla^{T(Z/B)})}=\int_B{\left\{\hat{A}(\nabla^{TB})\int_{Z/B}{\hat{A}(\nabla^{T(Z/B)})}\right\}}.$$
For simplicity, we assume the Dirac operator of the boundary family to be invertible; thus, the application of Theorem \ref{bc} yields 
$$\int_{Z/B}{\hat{A}(\nabla^{T(Z/B)})}=\left(ch(Ind)\right)_{dR}+\hat{\eta},$$
where the subscript indicates that we are dealing with a {\em de Rham representative} of the Chern character of the index bundle. Let us choose a virtual vector bundle $E$ with unitary connection $\nabla^E$, such that $[E]=[E_1\ominus E_2]=[Ind]$; then we may write $\left(ch(Ind)\right)_{dR}=ch(\nabla^E)+d\omega$, for some $\omega\in\Omega^{odd}(B)$, i.e.\ 
$$\int_Z{\hat{A}(\nabla^{\oplus})}=\int_B{\left\{\hat{A}(\nabla^{TB})\left(ch(\nabla^E)+d\omega+\hat{\eta}\right) \right\}};$$
finally, we apply Theorem \ref{aps} to the first summand and Stokes' theorem to the second  and  obtain:
$$\int_Z{\hat{A}(\nabla^{\oplus})}=Ind(/\!\!\!\!D\otimes E)+\left\{\xi(/\!\!\!\!D_{\partial B}\otimes E)+\int_{\partial B}\omega\hat{A}(\nabla^{TB})\right\}+\int_B{\hat{\eta}\hat{A}(\nabla^{TB})}.$$
Thus, we have an interpretation of the integral of the $\hat A$ genus form in terms of an index and contributions from each of the two faces;  the caveat is that this is just a formal result, in the sense that we had to introduce the form $\omega$ to store information we usually do not have access to. 
Rest assured however, that this information will not be needed for the computation of the $f$-invariant.

\subsection{Application to the $\check{f}$-invariant}

Let us extend the construction of the previous section to incorporate complex structures by making the following

\begin{df}\label{ufr2f}
A $(U,fr)^2_f$-manifold consists of
\begin{itemize}
\item{a $\langle2\rangle_f$-manifold Z,}
\item{a vector bundle $E$ over $B$ that turns $B$ into a $(U,fr)$-manifold,}
\item{a complex  vector bundle $E_2$ stably isomorphic to $T(Z/B)^{st}$ }
\item{a trivialization of the restriction of $E_2$ to the face $Z^{\prime\prime}\rightarrow B$,}
\item{and an isomorphism $E_2\oplus\pi^*E\cong TZ^{st}$ (in the stable sense).}
\end{itemize}
\end{df}
Again, we equip the complex bundles with hermitian metrics that restrict to the ones induced by the trivializations, and call unitary connections thereon compatible, if they restrict to the pure gauge ones. Then we have:

\begin{thm}\label{f_from_family}
Let $Z$ be a $(U,fr)^2_f$-manifold of dimension $2n+2$,  and let  $\nabla_1=\pi^*\nabla^{E}$ and $\nabla_2$  be compatible connections on $E_1=\pi^*E$ and $E_2$, respectively. Then the $\check{f}$-invariant  of its  corner $M$ is given by
$$\check{f}(M,\pi^*\nabla^{E}|_M,\nabla_2|_M)\equiv \int_B{\hat{e}_{\Gamma}\widetilde{Ell}(\nabla^{E})},$$
where we defined \textup{(}a de Rham representative of\textup{)} the $e_{\Gamma}$-invariant of  a family:
$$\hat{e}_{\Gamma}\equiv\int_{Z/B}Ell_0(\nabla_2)\mod\left(im(ch:K(B)\otimes\mathbb{Z}^{\Gamma}\rightarrow H^{even}(B,\mathbb{Q}[\zeta]))\right)_{dR}.$$
\end{thm}
\begin{proof}
First of all, we notice that $T(Z/B)$ inherits a natural orientation from $E_2$, so we may integrate over the fiber,
$$\int_Z{Ell_0\left(\nabla_2\right)\widetilde{Ell}\left(\pi^*\nabla^E\right)}=\int_B\left\{\widetilde{Ell}\left(\nabla^E\right)\int_{Z/B}Ell_0\left(\nabla_2\right)\right\}.$$
Then we observe that, by Theorem \ref{div}, $\int_Bch(\nabla^F)\widetilde{Ell}(\nabla^{E})$ yields a divided congruence for arbitrary hermitian vector bundles $F$ with unitary connection; finally, for arbitrary $\omega\in\Omega^{odd}\left(B\right)$, we have
$$\int_Bd\omega\widetilde{Ell}(\nabla^{E})=\int_{\partial B}\omega\widetilde{Ell}(\nabla^{E})=0$$ by Stokes' theorem and the flatness of $\nabla^E$ restricted to the boundary.
\end{proof}

\begin{rmk}
The rationale behind the definition of $\hat{e}_{\Gamma}$ is to exhibit an analogy as close as possible to the $e$-invariant, thus paving the way for Corollary \ref{formula}. However, the results of \cite{bunke-2007}\footnote{Thanks to U.\ Bunke for pointing out this reference} (which I was unaware of until recently) show that one can actually define an $e$-invariant of a family of framed manifolds  in terms of  smooth $K$-theory.
\end{rmk}

Take note that $E_2$ also induces  a canonical $Spin^{\mathbb{C}}$ structure on the vertical tangent. Thus, given a metric and connection on $T(Z/B)$ (which we assume to be of product type near the face $Z''$), we can construct  a family of Dirac operators  coupled to the formal vector bundle $(1-\zeta)^{-rkE_2}\Lambda_{-\zeta}E_2^*$; let us denote this family by $\eth^{\Gamma}$. 
\begin{cor}\label{formula}
Let $Z$ be as in Theorem \ref{f_from_family}, and let $\eth^{\Gamma}$ be constructed as above. If  the kernel of $\eth^{\Gamma}_{\partial X}$ is of constant rank along the fibers, then the $\check{f}$-invariant of the corner $M$ is given by
$$\int_B{\left\{\left(\hat{\eta}\left(\eth^{\Gamma}_{\partial X}\right)+{\textstyle{\frac{1}{2}}}ch\left(\nabla^{\ker\eth^{\Gamma}_{\partial X}}\right)+{\textstyle{\int_{Z''/B}{cs}}}\right) \widetilde{Ell}(\nabla^{E})\right\}}.$$
In particular, this expression is calculable from the face $Z''$.
\end{cor}
\begin{proof}
We replace $\nabla_2^{\mathbb{R}}$ by $\nabla^{T(Z/B)}$ in the $\hat{A}$ genus form underlying $Ell_0$,  hence giving rise to a Chern-Simons term of the form \eqref{chernsimons}, 
but clearly,
$$\int_{Z/B}d\ cs\equiv\int_{Z''/B}cs\mod im\left(d:\Omega^{odd}(B)\rightarrow\Omega^{even}(B)\right);$$
thus, taking into account the indeterminacy of $\hat{e}_{\Gamma}$, the application of Theorem \ref{bc}  yields the claim. 
\end{proof}

\begin{rmk}
We would like to point out that we can actually obtain a formula that does not require the kernel of $\eth^{\Gamma}_{\partial X}$ to form a vector bundle: In \cite{melrosepiazzafamily}, Melrose and Piazza prove a general index theorem for families with boundary by making use of $b$-calculus techniques and introducing the notion of a {\em{spectral section}} $P$. The formal application of their theorem to our situation is straightforward, but  since their result reduces to Theorem \ref{bc} under the aforementioned assumption (which will indeed be satisfied in the examples considered in this thesis), we are not going to elaborate on this.
\end{rmk}

In fact, Corollary \ref{formula} gives an index theoretical formula for the {\em topological} $f$-invariant, for we have:

\begin{prop}\label{equality}
Let $Z$ be as in Theorem \ref{f_from_family}. Then we have
$$\int_B{\hat{e}_{\Gamma}\widetilde{Ell}(\nabla^{E})}\in D_{n+1}^{\Gamma}\otimes\mathbb{Q/Z},$$
and this gives a representative of the topological $f$-invariant of $M$.
\end{prop}
\begin{proof} The Chern-Weil forms constructed from the curvature of a compatible connection $\nabla_i$ represent characteristic classes of  $E_i$ relative to  $\partial_iZ$. 
Since $\nabla_1=\pi^*\nabla^{E}$ comes from  the base, we have
$$\int_Z{Ell_0(\nabla_2)\widetilde{Ell}(\nabla_1)}=\int_Z{(Ell_0(\nabla_2)-1)\widetilde{Ell}(\nabla_1)};$$integrating over the fiber, we observe that the closed differential form  on $B$ given by $\int_{Z/B}(Ell_0(\nabla_2)-1)$ has rational periods, i.e.\  it represents a  class in $H^*(B,\mathbb{Q}[\zeta])$. On the other hand, the differential form $\widetilde{Ell}(\nabla^E)$ can be considered as a representative of a class in $H^{even}(B,\partial B;\mathbb{Q})\otimes D^{\Gamma}$. Thus, the integral  yields a rational combination of divided congruences, whereas the indeterminacy in $\hat e_{\Gamma}$ manifests itself in true (i.e.\ integral) divided congruences (by Theorem \ref{div}). Finally, we observe that
\begin{equation*} \begin{split}
 & \int_Z{(Ell_0(\nabla_2)-1)\widetilde{Ell}(\nabla_1)}\\    = & \int_Z{(Ell_0(\nabla_2)-1)(Ell(\nabla_1)-1)}- \int_Z{(Ell_0(\nabla_2)-1)(Ell_0(\nabla_1)-1)},\\ 
\end{split}
\end{equation*}
and the same reasoning as above shows that the second expression takes values in $\mathbb{Q}[\zeta]$, whereas the first expression can be identified with \eqref{f}.
\end{proof}
\begin{rmk}
We would like to stress that it is the fact that  $Ell(\nabla_1)$ comes from the base that allows us to show that the integrals above take rational values instead of real ones. 
\end{rmk}
Summarizing, we see that the distinction between $\check{f}$ and $f$ becomes negligible in the context of $(U,fr)^2_f$-manifolds: The former comes with a natural lift to the latter, and this property is respected by our formul\ae\ (Theorem \ref{f_from_family} and Corollary \ref{formula}).

\subsection{A vanishing theorem}

It can be shown {\em algebraically} that the topological $f$-invariant vanishes on framed manifolds which are in third filtration, i.e.\ which lift to $(U,fr)^3$-manifolds \cite{Lau00}. Here we shall provide geometrical insight by considering the following situation:

\begin{df}
A $(U,fr)^3_f$-manifold is a $\langle3\rangle$-manifold $Y$ that is a fiber bundle over a $(U,fr)^2_f$-manifold $B'$ where the typical fiber is a compact $\langle1\rangle$-manifold; this time, we do {\em not} require  $B'$ to be of even dimension. In addition, there is a complex vector bundle $E_3\cong T(Y/ B')^{st}$ that is trivialized over the face $\partial_3Y\rightarrow B'$, and there is an isomorphism $E_3\oplus\pi^*TZ^{st}\cong TY^{st}$.
\end{df}

\begin{thm}\label{vanish}
Let $M$ be the codimension-three corner of a $(U,fr)^3_f$-manifold $Y$. Then $\check{f}(M)\equiv0$. 
\end{thm}
\begin{proof}
First of all, we know that $\check{f}$ depends only on the corner and its split framing; furthermore, by Proposition \ref{stable}, all the possible splittings induced by the $(U,fr)^3_f$-structure yield congruent results. Thus, we may compute $\check{f}(M)$ from $Z=\pi^{-1}(\partial_1B')$, which we endow with its obvious $(U,fr)^2_f$-structure. Successive integration along the fibers yields
\begin{equation*}
\begin{split}
\check{f}(M,\pi^*(\nabla_1\oplus\nabla_2)|_M,\nabla_3|_M)&\equiv\int_ZEll_0(\nabla_3)\widetilde{Ell}(\pi^*(\nabla_2\oplus\nabla_1))\\
&=\int_{\partial_1B'}\left\{\widetilde{Ell}(\nabla_1\oplus\nabla_2)\int Ell_0(\nabla_3)\right\}\\
&=\int_{\partial_1B'}\left\{\widetilde{Ell}(\nabla_2)\int Ell_0(\nabla_3)\right\}\\
&=\int_{\partial B}\left\{\int\left\{\widetilde{Ell}(\nabla_2)\int Ell_0(\nabla_3)\right\}\right\},\\
\end{split}
\end{equation*}
where we made use of the fact that $\nabla_1$ is flat on $Z$. By assumption, the integrand, a closed form, extends over $B$, so the integral vanishes by Stokes' theorem.
\end{proof}

\section{The $f$-invariant of cartesian products}\label{productsection}

Now it is the time to illustrate the preceding ideas:  Let $X$, $B$ be $(U,fr)$-manifolds of even dimension. Then the cartesian product $Z=B\times X$ becomes a $(U,fr)^2_f$-manifold in the obvious way; furthermore, since $Z$ is a trivial fiber bundle, $\hat{e}_{\Gamma}$ is concentrated in degree zero, hence a constant function $e_{\Gamma}$ on $B$, so the formula of Theorem \ref{f_from_family} simplifies to
\begin{equation}\label{simpleproduct1}
\check{f}(\partial B\times\partial X)\equiv e_{\Gamma}\int_B{\widetilde{Ell}}(\nabla^{E}).
\end{equation}
Of course, the kernel of the operator $\eth^{\Gamma}_{\partial X}$ forms a trivial vector bundle on $B$, so Corollary \ref{formula} applies; we can do better though:

\begin{lem}\label{ereduced}
Let $X$ be a $(U,fr)$-manifold, and let
$$e_{\Gamma}(\partial X)\equiv\langle Ell_0^{\Gamma}(TX),[X,\partial X]\rangle\mod\mathbb{Z}^{\Gamma}.$$
Then we have
$$e_{\Gamma}(\partial X)\equiv e_{\mathbb{C}}(\partial X)\mod\mathbb{Z}^{\Gamma}.$$
\end{lem}
\begin{proof} We equip the hermitian vector bundle $E\cong TX^{st}$ with a unitary connection $\nabla^E$ preserving the trivialization of $E|_{\partial X}$; then the bundle  $\Lambda_{-\zeta}E^*$ inherits a connection $\nabla^{\Lambda_{-\zeta}E^*}$ preserving the induced trivialization, and by Lemma \ref{keylemma} we have
$$\int_XTd(\nabla^E)ch\left(\nabla^{\Lambda_{-\zeta}E^*}\right)\equiv(1-\zeta)^{rkE}\int_XTd(\nabla^E)\mod\mathbb{Z}[\zeta].$$
Stabilizing, i.e.~multiplying by $(1-\zeta)^{-rkE}\in\mathbb{Z}^{\Gamma}$, the claim follows.
\end{proof}

This enables us to establish the following remarkable result:

\begin{thm}\label{f_from_e}
Let $Y_1$, $Y_2$ be odd-dimensional framed manifolds, and let $m(Y_i)$ be {\em any} modular form of weight  $(\dim Y_i +1)/2$ w.r.t.~the fixed congruence subgroup $\Gamma=\Gamma_1(N)$ such that $\bar{m}(Y_i)=m(Y_i)-e_{\mathbb{C}}(Y_i)\in\mathbb{Z}^{\Gamma}[\![q]\!]$. Then we have
$$\check{f}(Y_1\times Y_2)\equiv \bar{m}(Y_1) e_{\mathbb{C}}(Y_2)\equiv-\bar{m}(Y_2) e_{\mathbb{C}}(Y_1).$$
In particular, the $\check{f}$-invariant of a product is antisymmetric under exchange of the factors.
\end{thm}
\begin{proof}
Combining \eqref{simpleproduct1} and Lemma \ref{ereduced}, we know that
$$\check{f}(Y_1\times Y_2)\equiv \bar{m}^{\prime}(Y_1) e_{\mathbb{C}}(Y_2),$$
where, by Example \ref{divex}, $\bar{m}^{\prime}(Y_1)$ is a divided congruence of the form $m^{\prime}-e_{\Gamma}(Y_1)$ for a rational homogeneous modular form $m^{\prime}$ of weight $(\dim Y_1 +1)/2$. If we choose {\em{any}} divided congruence $\bar{m}(Y_2)=m(Y_2)-e_{\mathbb{C}}(Y_2)$ for a suitable homogeneous modular form $m(Y_2)$ of level $N$ and weight $(\dim Y_2 +1)/2$, we obtain:
$$\begin{array}{rcccl}\check{f}(Y_1\times Y_2)&\equiv&\bar{m}^{\prime}(Y_1) e_{\mathbb{C}}(Y_2)&\equiv&\bar{m}^{\prime}(Y_1)m(Y_2)
\\ &\equiv
& -e_{\Gamma}(Y_1)m(Y_2) &
 \equiv& -e_{\Gamma}(Y_1)\bar{m}(Y_2)\\ & \equiv& -e_{\mathbb{C}}(Y_1)\bar{m}(Y_2)&\equiv&-m(Y_1)\bar{m}(Y_2)\\
 &\equiv&\bar{m}(Y_1)e_{\mathbb{C}}(Y_2)& &\end{array}.$$
Take note that Lemma \ref{ereduced} and Theorem \ref{div} ensure the existence of $\bar{m}(Y_i)$.
\end{proof}

This immediately implies:
\begin{cor}\label{e0f0}
Let $M$ be the cartesian product of two odd-dimensional framed factors, one of which has vanishing $e_{\mathbb{C}}$-invariant. Then $\check{f}(M)\equiv0$.
\end{cor}

\begin{rmk} While this result is similar to Theorem \ref{vanish}, it does not require any geometrical assumptions concerning $(U,fr)^3$-structures. \end{rmk}

\subsection{Sample calculations at level three}

Due to Theorem \ref{f_from_e}, it is quite simple to determine a {\em representative} of the $f$-invariant of a product, but it still has to be checked whether $f$ is non-trivial, which requires explicit calculations using divided congruences. Throughout this section, we fix $\Gamma=\Gamma_1(3)$, as $N=3$ is the smallest level at which two is not inverted. The ring of  modular forms for $\Gamma_1(3)$ is generated by
$$ E_1 = 1+6\sum_n \sum_{d|n} (\frac{d}{3})q^n,\quad\ E_3 = 1-9\sum_n \sum_{d|n} d^2(\frac{d}{3})q^n,$$
which are of weight one and three, respectively; we refer to the appendix for more details.

A word on notation: Although by Proposition \ref{equality} the distinction is  unnecessary,
 we are still going to use the notation $\check{f}$ in this section, and indicate the framed manifold it is computed from; if we write $f$, its argument will be given as the underlying element in the stable stems, which we denote by names prevalent in the literature, see e.g.~\cite{rave}.

\begin{prop}\label{nusquared}
Let $Y$ and $Y^{\prime}$ be framed manifolds that represent a generator $\nu$ of $\pi^{st}_3\cong\mathbb{Z}/24$. Then we have 
$$\check{f}(Y\times Y^{\prime})\equiv\frac{1}{2}\left(\frac{E_1^2-1}{12}\right)^2,$$
and this is non-trivial in $\underline{\underline{D}}_{4}^{\Gamma}\otimes\mathbb{Q/Z}$.
\end{prop}
\begin{proof}
We consider the  sphere $S^3$ as the sphere bundle $S(L)$ of the Hopf line $L$ over $S^2$; framing the base and the vertical tangent in a straightforward manner, we obtain a framing for the total space. By evaluating the (relative) Todd genus on the associated disk bundle, we  compute (see also Remark \ref{e_rmk}) that $e_{\mathbb{C}}(S(L))$ is given by $-\frac{1}{12}$; in view of Remark \ref{r_vs_c}, the $e_{\mathbb{R}}$-invariant must be either $-\frac{1}{24}$ or $\frac{11}{24}$, thus we conclude that $S(L)$ represents $\nu$. We apply Theorem \ref{f_from_e},  choose $$-\bar{m}(\nu)=\frac{E_1^2-1}{12}=\sum_{n=1}^{\infty}\left(\sum_{3\nmid d|n}d\right)q^n\in\mathbb{Z}[\![q]\!],$$
and compute
$$\frac{1}{12}\frac{E_1^2-1}{12}\equiv-\frac{1}{2}\left(\frac{E_1^2-1}{12}\right)^2\equiv\frac{1}{2}\left(\frac{E_1^2-1}{12}\right)^2.$$
To see that this is non-trivial in $\underline{\underline{D}}_4^{\Gamma}\otimes\mathbb{Q/Z}$, we have to compare the $q$-expansion of
$$f(\nu^2)\equiv\frac{1}{2}\left(\frac{E_1^2-1}{12}\right)^2=\frac{1}{2}q^2+3q^3+\frac{11}{2}q^4+O(q^5)$$
to those of a suitable basis of $M_4^{\Gamma}\otimes\mathbb{Q}$. A convenient choice consists of  $G_4^*$ and $G_4$, since the $q$-expansions of the latter two agree on powers of $q$ not divisible by three; but looking at $G_4-\textstyle{\frac{1}{240}}=q+9q^2+O(q^3)$, we immediately deduce that $f(\nu^2)$ is non-trivial.
\end{proof}

\begin{prop}\label{sigmasquared}
Let $Y$ and $Y^{\prime}$ be framed manifolds that represent a generator $\sigma$ of $\pi^{st}_7\cong\mathbb{Z}/240$. Then we have 
$$\check{f}(Y\times Y^{\prime})\equiv\frac{1}{2}\left(\frac{E_4-1}{240}\right)^2,$$
and this is non-trivial in $\underline{\underline{D}}_{8}^{\Gamma}\otimes\mathbb{Q/Z}$.
\end{prop}
\begin{proof}
Similar to the preceding case, the sphere $S^7$, considered as  the sphere of the quaternion line over $S^4$, represents $\sigma$, see e.g.~\cite{cf66}. Then (at least up to sign conventions) $e_{\mathbb{C}}(\sigma)=\frac{1}{240}$, so  we choose $\bar{m}(\sigma)=\frac{1}{240}(E_4-1)$ and `complete the square'. Thus, 
$$f(\sigma^2)\equiv\frac{1}{2}\left(\frac{E_4-1}{240}\right)^2=\frac{1}{2}q^2+9q^3+\frac{137}{2}q^4+325q^5+1175q^6+O(q^7),$$
and it is straightforward to show that this is non-trivial  by making use of the first, second, third, and sixth coefficients of
\begin{equation*}
\begin{split}
G_8^*+\frac{1093}{240}&=q+129q^2+q^3+16513q^4+78126q^5+129q^6+O(q^7)\\
\frac{E_8-1}{480}&=q+129q^2+2188q^3+16513q^4+78126q^5+282252q^6+O(q^7)\\
\frac{E_1^8-1}{48}&=q+21q^2+253q^3+1933q^4+9870q^5+35553q^6+O(q^7).
\qedhere\end{split}
\end{equation*}
\end{proof}

Recall from \cite{j4} that there is an 8-periodic family in the stable stems generalizing $\eta\in\pi^{st}_1$. Although we are not going to give an explicit representative, we denote by $\mu_k$ any framed manifold such that $[\mu_k]\in\pi^{st}_{8k+1}$ represents a member of this family (take note that our indexing differs from Adams').

\begin{prop}\label{musquared}
For any representatives $\mu$ of this family, we have
$$\check{f}(\mu_k\times\mu_l)\equiv\frac{1}{2}\frac{E_1-1}{2}.$$
\end{prop}
\begin{proof}
According to \cite{j4}, we have $e_{\mathbb{C}}(\mu_k)=1/2$, and this is the only possible non-trivial value of $e_{\mathbb{C}}$ in dimension $8k+1$. Thus, we judiciously choose 
$$\check{f}(\mu_k\times\mu_l)\equiv\frac{1}{2}\frac{E_1^{4k+1}-1}{2},$$
but
$$\frac{1}{2}\frac{E_1^{4k+1}-1}{2}-\frac{1}{2}\frac{E_1-1}{2}=\frac{E_1^{4k}-1}{4}E_1\in\mathbb{Z}[\![q]\!].$$
\end{proof}

\begin{rmk}\label{musquaredrmk}
Take note that Proposition \ref{musquared} just shows that the $f$-invariant of all products of the form $\mu_k\times\mu_l$ admits a universal representative;  in low dimensions, non-triviality of this representative is readily verified by hand (cf.~Example \ref{etasquared}), but we do not know how to establish this in full generality. On the other hand, it is known that these products represent non-trivial elements in $\pi^{st}_{8(k+l)+2}$: As shown in \cite{j4}, $e_{\mathbb{C}}(\mu_k)=1/2$ is equivalent to $d_{\mathbb{R}}(\mu_k)\neq0$, and, by the properties of the degree, then also $d_{\mathbb{R}}(\mu_k\times\mu_l)\neq0$.
\end{rmk}

\begin{ex}\label{etasquared}
It is easy to check that $[\mu_0]=\eta\in\pi^{st}_1\cong\mathbb{Z}/2$ can be represented by the circle with its non-bounding framing: Using a Fourier decomposition, we see that the $Spin^{\mathbb{C}}$ Dirac operator has symmetric spectrum and a single zero mode, so we conclude  $e_{\mathbb{C}}(\eta)=\frac{1}{2}$. We also see that
$$f(\eta^2) \equiv \frac{1}{2}\frac{E_1-1}{2} \equiv\frac{1}{2}\sum_n \sum_{d|n} (\frac{d}{3})q^n=\frac{1}{2}q+\frac{1}{2}q^3+O(q^4)$$
is non-trivial in $\underline{\underline{D}}_2^{\Gamma}\otimes\mathbb{Q/Z}$, since its $q$-expansion obviously cannot be congruent to a rational multiple of $({E_1^2-1})/{12}=q+3q^2+O(q^3)$.
\end{ex}

The images of $e_{\mathbb{C}}$ in dimensions $8k+3$, $8k+7$ are known  to be isomorphic to cyclic groups of order $d_{4k+2}$, $2d_{4k+4}$, respectively, where $d_{2k}$ denotes the denominator of $B_{2k}/2k$ (\cite{j4}, see also \cite{cf66}). The elements of $\pi^{st}_{8k+3}$, $\pi^{st}_{8k+7}$  on which these values are attained  lie in the image of the so-called {\em $J$-homomorphism},
$$J:\pi_rSO\rightarrow\pi^{st}_r,\quad r\geq1.$$

\begin{prop}\label{munu}
Let $M^{8k+3}$ be a representative of the generator of $Im(J)$ in dimension $8k+3$. Then
$$\check{f}(M^{8k+3}\times\mu_t)\equiv0.$$
\end{prop}
\begin{proof} Without loss of generality we may assume that $e_{\mathbb{C}}(M^{8k+3})$ is represented by $B_{4k+2}/(4k+2)$, so, for $k>0$, we have 
$$\check{f}(M^{8k+3}\times\mu_t)\equiv{\frac{1}{2}\frac{B_{4k+2}(1-E_{4k+2})}{4k+2}}=\sum\sigma_{4k+1}(n)q^n,$$
whereas for $k=0$ we have
\begin{equation*}
\begin{split}
\check{f}(M^{3}\times\mu_t)&\equiv{{\frac{1}{2}}}\frac{E_1^2-1}{12}\equiv{\frac{1}{2}}\left\{\frac{E_1^2-1}{12}+E_1^{4t}E_3-1\right\}\\
&\equiv{\frac{1}{2}}\left\{\frac{E_1^2-1}{12}+E_3-1\right\},\\
\end{split}
\end{equation*}
which expands integrally by Proposition \ref{E_1squaredE_3}.
\end{proof}

\begin{prop}\label{musigma}
 Let $M^{8l-1}$ be a representative of the generator of $Im(J)$ in dimension $8l-1$. Then
$$\check{f}(M^{8l-1}\times \mu_{t})\equiv\frac{1}{2}\frac{E_4-1}{16}.$$
\end{prop}

\begin{proof}
The theorem of von Staudt allows the computation of the denominator of the Bernoulli numbers (see e.g.~\cite{apostol}). More precisely, let $d_{2k}$ denote the denominator of $B_{2k}/2k$; if $2^i|2k$, then $2^{i+1}|d_{2k}$, and this result is sharp. In particular, this implies that for $l=(2n+1)2^m$, $2^{m+4}|2d_{4l}$ (and $2d_{4l}$ is precisely the order of $Im(J)$ in dimension $8l-1$).  Writing
\begin{equation*}
\begin{split}
\tilde{G}_{4l}&=\frac{B_{4l}}{8l}\left(1-E_{4l}\right),\\
\frac{B_{4l}}{8l}&=\frac{n_{4l}}{2d_{4l}}=\frac{n_{4l}}{2^{m+4}u_{4l}},\\
\end{split}
\end{equation*}
we have
$$\frac{1}{2}\tilde{G}_{4l}\equiv \frac{1}{2}\tilde{G}_{4l}+\frac{u_{4l}-1}{2}\tilde{G}_{4l}=\frac{1}{2}\frac{n_{4l}(1-E_{4l})}{2^{m+4}}.$$
With the help of Lemma \ref{E_4power} we compute that, modulo $\underline{\underline{D}}_{4(l+t)+1}$,  
$$-\frac{n_{4l}}{2^{m+5}}E_{4l}\equiv\frac{n_{4l}}{2^{m+5}}E_1^{4t+1}\equiv-\frac{n_{4l}}{2^{m+5}}E_4^l\equiv\frac{1}{2}\frac{E_4^l-1}{2^{m+4}},$$
from which it follows that
$$\check{f}(M^{8l-1}\times \mu_{t})\equiv\frac{1}{2}\frac{E_4^l-1}{2^{4+m}}.$$
Now let $l=(2n+1)2^m$, $l^{\prime}=(2n^{\prime}+1)2^m$, $n\leq n^{\prime}$. Then
$$\frac{1}{2}\frac{E_4^{l^{\prime}}-E_4^l}{2^{4+m}}=\frac{E_4^{l^{\prime}-l}-1}{2^{4+m+1}}E_4^l\in\mathbb{Z}[\![q]\!],$$
which means we can always reduce to the situation $l=2^m$; but we also have
$$\frac{1}{2}\frac{(E_4^{2^m}-1)-2^m(E_4-1)}{2^{4+m}}\in\mathbb{Z}[\![q]\!],$$
thus proving the claim.
\end{proof}

\begin{rmk}
This  situation  is akin to Proposition \ref{musquared}/Remark \ref{musquaredrmk}: From \cite{j4} (see also \cite{rave}), we know that $\eta x_{2l}$, where $x_{2l}$ denotes the generator of $Im(J)$ in dimension $8l-1$, is the generator of $Im(J)$ in dimension $8l$, hence non-trivial in $\pi^{st}_{8l}$.
\end{rmk}

\begin{ex}\label{etasigma}
The lowest-dimensional example can be realized geometrically by considering the spheres $S^7$ and $S^1$ with framings that represent $\sigma$ and $\eta$, respectively. In $\underline{\underline{D}}_5^{\Gamma}\otimes\mathbb{Q/Z}$, we may modify the $q$-expansion to read
$$f(\eta\sigma)\equiv\frac{1}{2}\frac{E_4-1}{240}\equiv\frac{1}{2}\frac{E_4-1}{240}+\frac{1}{2}\frac{E_5^{\Gamma}-1}{3}=q-3q^2+\frac{29}{2}q^3+157q^4+O(q^5),$$
and now a quick comparison with the first three coefficients of
\begin{equation*}
\begin{split}
\frac{E_5^{\Gamma}-1}{3}&=q-15q^2+q^3+241q^4+O(q^5),\\
\frac{E_1^5-1}{6}&=5q+60q^2+365q^3+1205q^4+O(q^5),\\
\end{split}
\end{equation*}
shows that $f(\eta\sigma)$ cannot be completed to an integral $q$-expansion.
\end{ex}

\begin{prop}\label{nusigma}
Let $Y^{4k-1}$ represent a generator of Im(J) in dimension $4k-1$, $k>1$, and let $Y'$ represent $\nu$. Then we have
$$\check{f}(Y^{4k-1}\times Y')\equiv0.$$
\end{prop}
\begin{proof}
For $k=2l$ we may proceed as in the previous proof:
\begin{equation*}
\begin{split}
\check{f}(Y^{8l-1}\times Y')&\equiv\frac{1}{4}\tilde{G}_{4l}\equiv\frac{1}{4}\tilde{G}_{4l}\pm\frac{u_{4l}\mp1}{4}\tilde{G}_{4l}=\pm\frac{1}{4}\frac{n_{4l}(1-E_{4l})}{2^{m+4}}\\
&\equiv\mp\frac{n_{4l}}{2^{m+6}}E_1^2\equiv\pm\frac{n_{4l}}{4}\frac{E_4^l-1}{2^{m+4}}\equiv\pm \frac{n_{4l}}{4}\frac{(E_4-1)\cdot l}{2^{m+4}},\\
\end{split}
\end{equation*}
where the sign is chosen according to whether $u_{4l}\equiv\pm1\mod4$, and the last step follows from  \eqref{ichoosel} and \eqref{chooseE_4}.\newline  
Making use of  Proposition \ref{E_4/64}, we see that in $\underline{\underline{D}}_{4l+2}^{\Gamma}\otimes\mathbb{Q/Z}$:
\begin{equation*}
\begin{split}
\frac{E_4-1}{2^6}&\equiv\frac{E_1^2-1}{2^5}-\frac{1}{4}(E_3^2-1)-\frac{1}{8}(E_1^3E_3-1)\\
&\equiv\frac{E_1^2-1}{2^5}-\frac{1}{4}(E_4^{l-1}E_3^2-1)-\frac{1}{8}(E_4^{l-1}E_1^3E_3-1)\\
&\equiv\frac{E_1^2-1}{2^5}\equiv-\frac{E_4^l-1}{2^5}\equiv-l\frac{E_4-1}{2^5};\\
\end{split}
\end{equation*}
if $l$ is even, we are done, otherwise we have to iterate once.\newline\newline
Now let $k=2l-1>1$. Then we have
$$\check{f}(Y^{8l-5}\times Y')\equiv\frac{1}{4}2\tilde{G}_{4l-2}\equiv\frac{1}{2}\frac{n_{4l-2}(1-E_{4l-2})}{8}\equiv\frac{n_{4l-2}}{4}\frac{E_1^2-1}{4},$$
but the latter is seen to be congruent to a form of top weight,
\begin{equation*}
\begin{split}
\frac{E_1^2-1}{2^4}&\equiv\frac{E_4-1}{2^5}+\frac{1}{2}(E_3^2-1)+\frac{1}{4}(E_1^3E_3-1)\\
&\equiv\frac{E_4-1}{2^5}+\frac{1}{2}(E_1^{4l-6}E_3^2-1)+\frac{1}{4}(E_1^{4l-6+3}E_3-1)\\
&\equiv\frac{E_4-1}{2^5}\equiv l\frac{E_4-1}{2^{5+m}}\equiv\frac{E_4^l-1}{2^{5+m}},\\
\end{split}
\end{equation*}
thus vanishes in  $\underline{\underline{D}}_{4l}^{\Gamma}\otimes\mathbb{Q/Z}$.
\end{proof}

\begin{rmk}
We would like to point out that Propositions \ref{nusquared} through \ref{nusigma} may be thought of as an elliptic analogue of  \cite[Theorem 5.5.8.~(b)]{rave}.
\end{rmk}

\section{The $f$-invariant of principal circle bundles}\label{circlebundlesection}

Let $L$ be a hermitian line bundle with unitary connection $\nabla^L$ over a compact manifold $B$; restricting to the unit disk in the fibers, we get a bundle $Z=D(L)$. Furthermore, the connection $\nabla^L$ induces a splitting of the tangent bundle. Clearly, the vertical tangent bundle is
$$T(Z/B)\cong\pi^*L,$$
therefore it inherits a natural complex structure. Restricting to the sphere bundle $S(L)$, we have
$$(\pi^*L)|_{S(L)}\cong T(Z/B)|_{S(L)}\cong T(S(L)/B)\oplus S(L)\times\mathbb{R},$$
and we can trivialize the vertical tangent of the principal circle bundle $S(L)$, hence also $(\pi^*L)|_{S(L)}$. 
Thus, if $L$ is a hermitian line with connection over an even-dimensional $(U,fr)$-manifold $B$, 
 then $D(L)$ can be turned into a $(U,fr)^2_f$-manifold. 

Let us compute the $f$-invariant in this situation: Given any compatible  connections we may invoke Theorem \ref{f_from_family}; further simplification can be achieved by a result similar to Lemma \ref{ereduced}:

\begin{lem}\label{ezdisk}
In the situation of a disk bundle $Z=D(L)$, we may replace $Ell_0(\nabla_2)$ by $Td(\nabla_2)$ for the computation of the $\hat{e}_{\Gamma}$-form in Theorem \ref{f_from_family}.
\end{lem}
\begin{proof}
In terms of the $i\mathbb{R}$-valued curvature two-form $F^{\nabla_2}$, we have
$$Td\left(\nabla_2\right)-Ell_0\left(\nabla_2\right)=-\frac{\zeta}{1-\zeta}\frac{iF^{\nabla_2}}{2\pi};$$
upon integration along the fiber, the RHS takes values in $\mathbb{Z}^{\Gamma}$, and, 
when computing $f$, this integer will get multiplied by the divided congruence $\int_B{\widetilde{Ell}(\nabla^E)}$.
\end{proof}
In order to make contact with index theory, we endow the vertical tangent bundle with a metric $g^{T(Z/B)}$ such that it is of product type near the boundary and that the circle acts isometrically. Then we construct the connection $\nabla^{T(Z/B)}$ as in section \ref{classicalindex}; if we denote by $e$ the vertical unit tangent of the circle bundle, we see that
$\nabla^{T(Z/B)}\langle e,e\rangle=0$, which implies that 
$$\nabla^{T(Z/B)}e=0,$$
i.e.\ this connection is already trivializing (cf.\ \cite{zhang94}). 
Now we can prove:

\begin{prop}\label{eform}
In the situation of the disk bundle $D(L)$, the cohomology class of the $\hat{e}_{\Gamma}$-form is given by
$$[\hat{e}_{\Gamma}]\equiv\sum_{k=0}^{\infty} B_{k+1}\frac{\left(c_1(L)\right)^k}{(k+1)!}\mod im(ch:K(B)\otimes\mathbb{Z}^{\Gamma}\rightarrow H^{even}(B,\mathbb{Q}[\zeta])).$$
\end{prop}
\begin{proof}
By Lemma \ref{ezdisk}, it is sufficient to consider
$$\int_{Z/B}Td\left(\nabla_2\right)=\int_{Z/B}Td\left(\nabla^{T(Z/B)}\right)+\int_{Z/B}d\ cs\left(Td,\nabla^{T(Z/B)},\nabla_2\right);$$
analogously to Corollary \ref{formula}, the last summand can be reduced to an integral over $S(L)/B$, at least up to exact forms on the base.  Clearly, we have
$$cs\left(Td,\nabla^{T(Z/B)},\nabla_2\right)=cs(\hat{A},\nabla^{T(Z/B)},\nabla_2^{\mathbb{R}})\exp\left(\frac{iF^{\nabla_2}}{4\pi}\right),$$
and since the $\hat{A}$-genus form comes from an even power series, formula \eqref{transgress} yields
$$cs(\hat A)=\sum_{k=0}^{\infty}c_k\int_0^1\mbox{tr}\left(\omega F_t^{2k+1}\right)dt,$$
for explicitly calculable coefficients $c_k$; 
but  the $\mathfrak{so}(2)$-valued (i.e.\ abelian) curvature two-forms of $\nabla^{T(Z/B)}$ and $\nabla_2^{\mathbb{R}}$ vanish upon restriction to $S(L)$, hence so do $F_t$ and $cs$.\newline
Now,  by assumption, the circle acts isometrically on the fibers, which implies that the kernel of the  boundary family $\eth_{\partial}$ is induced by the (trivial) $S^1$-representation $\ker\eth_{S^1}$, hence it is of constant rank. Therefore, we may apply Theorem \ref{bc}, and since $ch(\ker\eth_{\partial})=1$, we have
$$[\hat{e}_{\Gamma}]\equiv[\hat{\eta}(\eth)]+{\textstyle\frac{1}{2}}\mod im(ch:K(B)\otimes\mathbb{Z}^{\Gamma}\rightarrow H^{even}(B,\mathbb{Q}[\zeta])).$$
Luckily, the $\hat{\eta}$-form of the Dirac operator for principal circle bundles has been computed in \cite{zhang94} and  \cite{goette00} (see also appendix \ref{etazhang}); its underlying cohomology class is 
\begin{equation}\label{sign2}
[\hat{\eta}(\eth)]=\sum_{k=1}^{\infty}B_{k+1}\frac{(c_1(L))^k}{(k+1)!},
\end{equation}
and the addition of ${\textstyle\frac{1}{2}}=B_1(1)$ completes the proof.
\end{proof}

\begin{cor}\label{fcircle}
Let $L$ be a hermitian line with unitary connection $\nabla^L$ over a $(U,fr)$-manifold $B$ of dimension $2n+2$. If we denote by $S(L)_|$ the  framed  circle bundle over $\partial B$, we have 
$$f(S(L)_|)\equiv\sum_{k=0}^{n} \frac{B_{k+1}}{(k+1)!}\int_B{\left\{\left(\frac{iF^L}{2\pi}\right)^k\widetilde{Ell}(\nabla^{E})\right\}}.$$
\end{cor}
\begin{proof}
The first Chern class of $L$ may be represented by the  normalized curvature two-form, so the result follows from Proposition \ref{equality},  Theorem \ref{f_from_family}, Proposition \ref{eform}, and the fact that powers $k>n$ cannot contribute (take note that $\widetilde{Ell}$ is concentrated in positive degree). \end{proof}
\begin{rmk}
Modulo the integers, $[\hat{e}_{\Gamma}]$ is an odd function of $c_1(L)$. Thus, if we replace $L$ by $L^*$ in Corollary \ref{fcircle}, we see that the $f$-invariant just changes sign. 
\end{rmk}

\begin{rmk}\label{e_rmk}
We may also use Proposition \ref{eform} to compute the $e_{\mathbb{C}}$-invariant of a principal circle bundle $S(L)$ over a closed framed base $B$ of dimension $2k$: Take note that if $E$ is a hermitian vector bundle with unitary connection, then the fact that $B$ is framed implies that $\int_Bch(\nabla^E)$ is an integer, the index of $\eth\otimes E$. Thus, modulo the integers, the evaluation of the relative Todd genus of the $(U,fr)$-manifold $D(L)$ reduces to the evaluation of $[\hat{\eta}(\eth)]+\textstyle{\frac{1}{2}}$ on $B$, which implies that
$$e_{\mathbb{C}}\left(S(L)\right)\equiv\frac{B_{k+1}}{k+1}\int_Bch\left(\nabla^L\right)\mod\mathbb{Z}.$$
In particular, this formula applies to the nilmanifolds of \cite{DSheisenberg}, cf.~\cite{goette00}.
\end{rmk}

\subsection{Calculations for torus bundles}

As an application of Corollary \ref{fcircle}, we consider the situation where the $(U,fr)$-manifold $B$ is a disk bundle itself: Let $L$, $L'$ be hermitian lines over a closed, framed manifold $B'$ of positive even dimension; we choose a connection on $L'$ and turn $D(L')$ into a $(U,fr)$-manifold $B$. Pulling back the line $L$, we proceed similarly to obtain a $(U,fr)^2_f$-manifold $D(\pi^*L)$. In particular, the corner $M$ is a principal torus bundle, which we refer to as the {\em double transfer} of $B'$ (with respect to the lines $L$, $L'$).

\begin{rmk}
The iterated transfers appear in Knapp's work investigating the Adams filtration of Lie groups \cite{kna78}, treating a (compact) Lie group $G$ as principal  bundle over $G/T$, where $T$ is a maximal torus. In particular, the tangent bundle of $G/T$ is stably trivial, and there is a convenient description of the cohomology of $G/T$ in terms of the roots \cite{BH1}.
\end{rmk}

\begin{lem}\label{ezdrop}
Let $\Gamma=\Gamma_1(3)$, let $L$ and $L'$ be hermitian lines over a closed framed manifold  $B'$ of dimension $2n>0$, and let $M$ be the double transfer constructed  above. Then, denoting the first Chern classes of $L$ and $L'$ by $x$ and $y$, respectively, the ${f}$-invariant of   $M$ is given by
$${f}(M)\equiv\sum_{k=1}^{n-1} \frac{B_{k+1}}{(k+1)!}\langle x^k\widetilde{Ell}(y)/y,[B']\rangle.$$
\end{lem}
\begin{proof}
Clearly, the pair $(D(L'),S(L'))$ may be identified with the Thom space of $L'$; furthermore, the tangent of the $(U,fr)$-manifold $B$ is stably isomorphic to $\pi^*L'$, so the Thom isomorphism yields
$$\sum_{k=0}^{n}{\textstyle \frac{B_{k+1}}{(k+1)!}}\langle x^k\widetilde{Ell}(y)/y,[B']\rangle=\sum_{k=0}^{n} {\textstyle \frac{B_{k+1}}{(k+1)!}}\langle (\pi^*x)^k\widetilde{Ell}(TB),[D(L'),S(L')]\rangle,$$
and the RHS is precisely the formula of Corollary \ref{fcircle}. If we interpret the summand for $k=0$ as
$$\frac{1}{2}\langle\widetilde{Ell}(y)/y,[B']\rangle=e_{\mathbb{C}}(S^1)\langle\widetilde{Ell}(y)/y,[B']\rangle\equiv \check{f}(S(L'))\times S^1),$$
the latter is easily seen to be congruent to zero by Propositions \ref{munu}, \ref{musigma}, and the fact that $e_{\mathbb{C}}(S(L'))$ is represented by an integer multiple of $B_{n+1}/(n+1)$, cf.~Remark \ref{e_rmk}. Similarly, the contribution proportional to $\langle x^n,[B']\rangle$ may be identified with $\check{f}(S^1\times S(L))\equiv 0$.
\end{proof}

\begin{rmk}
It should be noted that the explicit choice of framing for $B'$ plays only a minor r\^ole in the formula above; in particular, orientation-preserving reframings of $B'$ lead to a double transfer having the same $f$-invariant.
\end{rmk}

For the remainder of this section, we are going to fix $\Gamma=\Gamma_1(3)$ and use the notations of Lemma \ref{ezdrop}; we are going to use the latter  to compute the $f$-invariant of the {\em general} double transfer in a given dimension, thus enabling us to determine the precise conditions for non-triviality. Clearly, the double transfer on a two-dimensional base will have vanishing $f$-invariant; increasing dimensions, things become more interesting, starting with:

\begin{prop}
Let $B'$ be a closed framed manifold of dimension four. The double transfer w.r.t.\ the lines $L$, $L'$ has non-trivial $f$-invariant if and only if $\langle xy,[B']\rangle$ is odd, in which case it is given by
$$\frac{1}{2}\left(\frac{E_1^2-1}{12}\right)^2\in\underline{\underline{D}}^{\Gamma}_4\otimes\mathbb{Q/Z}.$$
\end{prop}
\begin{proof}
Making use of Lemma \ref{ezdrop} and the expansions \eqref{expandedgenus}, \eqref{expandedtodd}, we compute
 $f\equiv\frac{1}{12}\frac{E_1^2-1}{12}\langle xy,[B']\rangle$; thus, we can proceed as in Proposition \ref{nusquared}.
\end{proof}
\begin{ex} An obvious choice is to take $B^{\prime}$ to be $S^2\times S^2$, which we may think of as $Spin(4)/T$.  Taking $x$ and $y$ to be (minus) the generators of the cohomology of the respective factors, we essentially recover the situation of Proposition \ref{nusquared}.\end{ex}

\begin{prop}
Let $B'$ be a closed framed manifold of dimension six. The double transfer w.r.t.\ the lines $L$, $L'$ has non-trivial $f$-invariant if and only if $\langle xy^2,[B']\rangle$ is odd, in which case it is given by
$$\frac{1}{2}\left(\frac{E_1^2-1}{12}\right)^2\in\underline{\underline{D}}^{\Gamma}_5\otimes\mathbb{Q/Z};$$
furthermore, this differs from the situation of  Example \textup{\ref{etasigma}}.
\end{prop}
\begin{proof}
By Lemma \ref{ezdrop}, we just have to consider the coefficient of $xy^2$;  by \eqref{expandedgenus} and \eqref{expandedtodd}, it is given by
\begin{equation*}
\begin{split}
\quad\ \frac{1}{12}\frac{iE_1^3-iE_3}{18\sqrt{3}}&\equiv-\frac{i\sqrt{3}}{3}\left\{\frac{1}{2}\frac{E_1^2-1}{12}\frac{1-E_3}{9}+\frac{E_1^2-1}{12}\frac{E_1^3-1}{18}\right\}\\
&\equiv\frac{1}{2}\frac{E_1^2-1}{12}\frac{E_3-1}{9}\equiv\frac{1}{2}\left(\frac{E_1^2-1}{12}\right)^2\\
&=\frac{1}{2}q^2+3q^3+\frac{11}{2}q^4+O(q^5),\\
\end{split}
\end{equation*}
which is seen to be non-trivial by checking the first and fourth coefficients of
\begin{equation*}
\begin{split}
5\frac{E_5^{\Gamma}-1}{3}&=5q-75q^2+5q^3+1205q^4+O(q^5),\\
\frac{E_1^5-1}{6}&=5q+60q^2+365q^3+1205q^4+O(q^5);\\
\end{split}
\end{equation*}
obviously, this argument still holds true  if we add 
$$f(\eta\sigma)\equiv\frac{1}{2}\frac{E_4-1}{240}+\frac{1}{2}\frac{E_5^{\Gamma}-1}{3}=q-3q^2+\frac{29}{2}q^3+157q^4+O(q^5).$$
\end{proof}

\begin{ex}
We borrow an example from \cite{Lau00}, choosing  $B^{\prime}=SU(3)/T$ and $x$ and $y$ to be the simple roots. A straightforward computation yields $\langle xy^2,[B']\rangle=3$, hence establishing non-triviality.
\end{ex}

\begin{prop}
Let $B'$ be a closed framed manifold of dimension eight. Then the double transfer w.r.t.\ any lines $L$, $L'$ has trivial $f$-invariant in $\underline{\underline{D}}^{\Gamma}_6\otimes\mathbb{Q/Z}$.
\end{prop}
\begin{proof}
By Lemma \ref{ezdrop}, we know that $f$ is represented by the evaluation of 
$$\frac{1}{12}\frac{13(E_1^4-1)-16(E_1E_3-1)}{2160}y^3x-\frac{1}{720}\frac{E_1^2-1}{12}yx^3,$$
but
\begin{equation*}
\begin{split}
&\frac{1}{12}\frac{13(E_1^4-1)-16(E_1E_3-1)}{2160}=\frac{1}{12}\frac{5(E_1^4-1)-2(E_4-1)}{2160}\\
\equiv&\frac{1}{8}\frac{E_1^4-1}{8}-\frac{1}{2}\frac{E_4-1}{16}\equiv\frac{1}{8}\frac{E_1^4-1}{8}+{9}\left(\frac{E_1^2-1}{12}\right)^3-\frac{1}{2}\frac{E_4-1}{16}\\ \equiv&\frac{E_1^2-1}{2^6}-\frac{1}{2}\frac{E_4-1}{16}\equiv-\frac{E_4-1}{2^6}-\frac{1}{2}\frac{E_4-1}{16}=-\frac{3}{4}\frac{E_4-1}{16}\\
\end{split}
\end{equation*}
and
$$-\frac{1}{720}\frac{E_1^2-1}{12}\equiv\frac{1}{36}\frac{E_4-1}{240}\equiv\frac{1}{4}\frac{E_4-1}{16}$$
may be identified with (multiples of) $f(\nu\sigma)$, but the latter is trivial by Proposition \ref{nusigma}.
\end{proof}

\begin{rmk}
Revisiting the above at level {\em two}, an admittedly tedious calculation along the lines of this section shows that the corresponding $f$-invariant will be non-trivial if and only if $3\nmid\langle x^3y,[B']\rangle$, and $Spin(5)/T$ fits the bill; since this result also follows from  \cite{Lau00}, we do not bother with details.
\end{rmk}

\begin{prop}
Let $B'$ be a closed framed manifold of dimension ten. Then the double transfer w.r.t.\ any lines $L$, $L'$ has trivial $f$-invariant in $\underline{\underline{D}}^{\Gamma}_7\otimes\mathbb{Q/Z}$.
\end{prop}
\begin{proof}
By Lemma \ref{ezdrop}, it is sufficient to evaluate
$$\frac{i}{\sqrt{3}}\left(xy^4\frac{1}{12}\frac{E_1^2(E_1^3-E_3)}{216}-x^3y^2\frac{1}{720}\frac{E_1^3-E_3}{18}\right)$$
on $B'$. In $\underline{\underline{D}}^{\Gamma_1(3)}_7\otimes\mathbb{Q/Z}$, we may rewrite
\begin{equation*}
\begin{split}
\frac{1}{12}\frac{E_1^2(E_1^3-E_3)}{216}&\equiv-\frac{1}{4\cdot9}\left(\left(\frac{E_1^2-1}{4}\right)^2+\frac{E_4-1}{16}\right)\frac{E_1^3-E_3}{9},\\
\frac{1}{720}\frac{E_1^3-E_3}{18}&\equiv-\frac{1}{2\cdot3}\frac{E_4-1}{240}\frac{E_1^3-E_3}{9}.\\
\end{split}
\end{equation*}
By coupling the Dirac operator on $B'$ to $L'\otimes L\oplus L' \otimes L^*\ominus2L'$, we obtain the divisibility result $5!|\langle(20x^3y^2+10xy^4),[B']\rangle$, i.e.~$4|\langle(2x^3y^2+xy^4),[B']\rangle$;  interchanging $L$ and $L'$, we get $4|\langle(2y^3x^2+yx^4),[B']\rangle$. But a short calculation with Steenrod squares (see e.g.~\cite{steenrod}, \cite{milnorstasheff}) shows that $x^3y^2$ is already even: Recall that, on a framed manifold, $Sq^k$ vanishes on classes of codimension $k$; furthermore, $Sq^1$ (i.e.~the Bockstein) vanishes on integral classes, so we compute
$$0\equiv Sq^2(x^3y)\equiv Sq^2(x^3)y+x^3Sq^2(y)\equiv x^4y+x^3y^2\equiv x^3y^2\mod2.$$
Thus, $4|\langle  xy^4,[B']\rangle$, and consequently, $f$ admits an integral $q$-expansion.
\end{proof}

\begin{lem}\label{evencohomology}
Let $B'$ be a  closed, framed manifold of dimension twelve and let $x$, $y$ $\in H^2(B',\mathbb{Z})$. Then $\langle x^3y^3,[B']\rangle$ and $\langle xy^5,[B']\rangle$ are even.
\end{lem}
\begin{proof}
Let $L$, $L'$ be hermitian lines such that their first  Chern classes are $x$, $y$, respectively; we endow the lines with unitary connections and compute the index of several twisted Dirac operators to obtain the following divisibility results:
$$6!|\langle(x+y)^6+(x-y)^6-2x^6-2y^6,[B']\rangle$$
$$\Rightarrow8|\langle x^4y^2+x^2y^4,[B']\rangle,$$
$$6!|\langle(x+2y)^6+(2x+y)^6-x^6-y^6,[B']\rangle$$
$$\Rightarrow16|\langle12(x^5y+xy^5)+60(x^4y^2+x^2y^4),[B']\rangle.$$
Thus, we also have $8|\langle6(x^5y+xy^5),[B']\rangle$, and from
$$6!|\langle6(x^5y+xy^5)+20x^3y^3+15(x^4y^2+x^2y^4),[B']\rangle$$
we may now deduce 
$2|\langle x^3y^3,[B']\rangle$.
Finally, making use of Steenrod squares, we have
$$0\equiv Sq^4(x^3y)\equiv x^5y+x^4y^2\equiv x^5y+Sq^6(x^2y)\equiv x^5y\mod 2,$$
so $x^5y$ and, analogously, $xy^5$ are even.
\end{proof}

\begin{prop}\label{G_2}
Let $B'$ be a closed framed manifold of dimension twelve. The double transfer w.r.t.~lines $L$, $L'$ has non-trivial $f$-invariant if and only if $\textstyle{\frac{1}{2}}\langle xy^5,[B']\rangle$ is odd, in which case it is given by
$$\frac{1}{2}\left(\frac{E_1^2-1}{12}\right)^3\in\underline{\underline{D}}^{\Gamma}_8\otimes\mathbb{Q/Z};$$
furthermore, this differs from the situation of Proposition \textup{\ref{sigmasquared}}.
\end{prop}

{\em Note added for v2:} The attempt to streamline the proof of Proposition \ref{G_2} introduced an error in the previous version (v1); the present version (v2) remedies this by reverting to the original, correct form  of  Lemma \ref{E_1squaredvanishes} and the subsequent calculations used in this proof.

\begin{proof}
We apply Lemma \ref{ezdrop}, and by Lemma \ref{evencohomology}, we gain an extra factor of two; in particular, the coefficient of $\textstyle{\frac{1}{2}}\langle x^3y^3,[B']\rangle$ reads
\begin{equation*}
\begin{split}
&2\cdot\frac{1}{720}\frac{16(E_1E_3-1)-13(E_1^4-1)}{2160}
=-\frac{2}{3^3\cdot240}\frac{2(E_4-1)-5(E_1^4-1)}{240}\\
&\equiv\frac{2}{3^3}\left(\frac{E_4-1}{240}\right)^2+\frac{1}{3^4\cdot240}\frac{E_1^4-1}{8}
\equiv-\frac{1}{3^4\cdot8}\frac{E_4-1}{240}\equiv\frac{5}{3^3}\left(\frac{E_4-1}{240}\right)^2;\\
\end{split}
\end{equation*}
thus, we are left with the contributions
$$2\cdot\frac{121(E_1^6-1)-152(E_1^3E_3-1)+40(E_3^2-1)}{2^7\cdot3^6\cdot5\cdot7},$$
$$2\cdot\frac{1}{2^5\cdot3^3\cdot5\cdot7}\frac{E_1^2-1}{12},$$
corresponding to $\textstyle{\frac{1}{2}}\langle x^5y,[B']\rangle$ and $\textstyle{\frac{1}{2}}\langle xy^5,[B']\rangle$, respectively. 
Furthermore, the divisibility results of Lemma \ref{evencohomology} imply that we may write
$\langle x^5y,[B']\rangle=60k-\langle xy^5,[B']\rangle-\frac{20}{3}\frac{1}{2}\langle x^3y^3,[B']\rangle$; combined with
$$20\frac{E_1^2-1}{2^7\cdot5\cdot7}\equiv-\frac{1}{4}\frac{E_6-1}{7\cdot8}\equiv-\frac{E_1^2-1}{2^5}\equiv0,$$
where triviality follows from Lemma \ref{E_1squaredvanishes}, we end up with an expression for $f$ determined by $\langle xy^5,[B']\rangle$ alone. We now set $\langle xy^5,[B']\rangle=2$ (or, more generally, an odd multiple thereof), write $-2=18-20$, and simplify to arrive at
$$f\equiv\frac{121E_1^6-152E_1^3E_3+40E_3^2}{2^7\cdot3^6\cdot5\cdot7}\cdot2+\frac{E_1^2}{2^7\cdot3^4\cdot5\cdot7}\cdot18,$$
which is congruent to the desired result by Proposition \ref{g2_ugly_2}. Non-triviality is established by comparing the first, second and fourth coefficient of the $q$-expansion of
$$\frac{1}{2}\left(\frac{E_1^2-1}{12}\right)^3=\frac{1}{2}q^3+\frac{9}{2}q^4+O(q^5)$$
to those of
\begin{equation*}
\begin{split}
G_8^*+\frac{1093}{240}&=q+129q^2+q^3+16513q^4+O(q^5),\\ 
\frac{E_8-1}{480}&=q+129q^2+2188q^3+16513q^4+O(q^5),\\ 
\frac{E_1^8-1}{48}&=q+21q^2+253q^3+1933q^4+O(q^5);\\ 
\end{split}
\end{equation*}
this argument still applies if we add
$$f(\sigma^2)\equiv\frac{1}{2}\left(\frac{E_4-1}{240}\right)^2=\frac{1}{2}q^2+9q^3+\frac{137}{2}q^4+O(q^5).$$
\end{proof}

\begin{ex} We choose $B'$ to be the framed manifold $G_2/T$. Let us recall some facts about  its cohomology ring \cite{BS55}: Rationally, it is generated by classes $\alpha,\ \beta\in H^2(G_2/T,\mathbb{Z})$ subject to the relations
$$\alpha^2+3\beta^2+3\alpha\beta=0,\qquad\alpha^6=0=\beta^6;$$
furthermore, we have
$$\langle\alpha\beta^5,[G_2/T]\rangle=2,$$
which is precisely what we need, i.e.\ we take $x=\alpha$, $y=\beta$.  
\end{ex}

\begin{rmk}
It is of course possible to continue the  program initiated above, i.e.\ to calculate the $f$-invariant of  the generic double transfer systematically; however, as the computations become increasingly more involved, one should look for results complementing our approach. In fact, such results exist (at least partially): It is possible to compute the $f$-invariant {\em algebraically}, never leaving the context of the ANSS, and this has been done for several beta-elements in \cite{jensniko}. In particular, a straightforward comparison of their results to Proposition \ref{G_2} (making use of Lemma \ref{E_1squaredvanishes}) shows that the example above represents $\beta_3$ at the prime two. (Of course, this does not come as a surprise, since, by our Proposition \ref{G_2}, it cannot be $\sigma^2=\beta_{4/4}$.)
\end{rmk}

\cleardoublepage
\addtocontents{toc}{\protect\contentsline{section}{Appendix}{}}
\appendix
\section{Useful formul\ae}

\subsection*{\dots from analytic number theory}

(See e.g.~\cite{apostol}) The {\em Hurwitz zeta function}  is defined by analytic continuation of the series $\zeta(s,x)=\sum_{n=0}^{\infty}(n+x)^{-s}$, $x>0$ and $s>1$. Take note that $\zeta(s,1)=\zeta(s)$ is the usual Riemann zeta function. It satisfies the functional equation
$$\zeta\left(1-s,\frac{m}{n}\right)=\frac{2\Gamma(s)}{(2\pi n)^s}\sum_{k=1}^{n}{\cos\left(\frac{\pi s}{2}-\frac{2\pi km}{n}\right)\zeta\left(s,\frac{k}{n}\right)}$$
for integers $1\leq m\leq n$. We also have the relation
$$\zeta(-n,x)=-\frac{B_{n+1}(x)}{n+1},$$
where the  Bernoulli polynomials are given by $$\frac{te^{xt}}{e^{t}-1}=\sum_{k=0}^{\infty}{B_k(x)\frac{t^k}{k!}};$$
they  satisfy
$$B_n(1-x)=(-1)^nB_n(x)$$
and 
$$B_n(mx)=m^{n-1}\sum_{k=0}^{m-1}{B_n\left(x+\frac{k}{m}\right)}.$$
If no argument is indicated, it is understood to be one, i.e.\ $B_n=B_n(1)$.

\subsection*{\dots involving modular forms}

Let $\Gamma\subset SL(2,\mathbb{Z})$ be a subgroup of finite index, e.g.~the following {\em congruence subgroups} of level $N>1$ \cite{HBJ}, \cite{schoeneberg}:
\begin{equation*}
\begin{split}
\Gamma_0(N)&=\left\{\gamma\in SL(2,\mathbb{Z})\ |\ \gamma\equiv\left(\begin{array}{cc} * & * \\ 0 & * \end{array}\right)\mod N\right\},\\
\Gamma_1(N)&=\left\{\gamma\in SL(2,\mathbb{Z})\ |\ \gamma\equiv\left(\begin{array}{cc} 1 & * \\ 0 & 1 \end{array}\right)\mod N\right\};\\
\end{split}
\end{equation*}
their index  is given by
$$[SL(2,\mathbb{Z}):\Gamma_1(N)]=N^2\prod_{p|N}(1-p^{-2}),$$
$$[SL(2,\mathbb{Z}):\Gamma_0(N)]=N\prod_{p|N}(1+p^{-1}).$$
Now let $k\geq0$ be an integer. Recall that a function $f:\mathfrak{h}\rightarrow\mathbb{C}$ is called a {\em modular form of weight $k$ w.r.t.~$\Gamma$} if:
\begin{itemize}
\item[(i)]{$f$ is holomorphic on $\mathfrak{h}$,}
\item[(ii)]{for all $\gamma=\left(\begin{array}{cc} a & b \\ c & d \end{array}\right) \in\Gamma$, we have:
$(c\tau+d)^kf(\tau)=f\left(\frac{a\tau+b}{c\tau+d}\right)$},
\item[(iii)]{and for every $S=\left(\begin{array}{cc} a & b \\ c & d \end{array}\right)\in SL(2,\mathbb{Z})$, $(c\tau+d)^{-k}f\left(\frac{a\tau+b}{c\tau+d}\right)$ admits a Fourier expansion of the form $\sum_{n\geq0}a_nq^{n/N}$.}
\end{itemize}
The vector space of modular forms of a given weight $k$ is {\em finite-dimensional}; the valence formula implies the following upper bound:
\begin{equation}\label{valence}
\dim M_k(\Gamma)\leq1+\frac{k}{12}[PSL(2,\mathbb{Z}):P\Gamma].
\end{equation}
Take note that if $f(\tau)$ is a modular form w.r.t.~$SL(2,\mathbb{Z})$, then $g(\tau)=f(N\tau)$ is a modular form w.r.t.~$\Gamma_0(N)$:
\begin{proof}
Let $N|b$ and ${{\left(\begin{array}{cc} a & b \\ c & d \end{array}\right) \in SL(2,\mathbb{Z})}}$. Then we use $$\left( \begin{array}{cc} 1/N & 0 \\ 0 & 1 \end{array}  \right)\left( \begin{array}{cc} a & b \\ c & d \end{array}  \right)\left( \begin{array}{cc} N & 0 \\ 0 & 1 \end{array}  \right)=\left( \begin{array}{cc} a & b/N \\ Nc & d \end{array}  \right)$$ to obtain any element in $\Gamma_0(N)$. Assuming $f$ to be of weight $k$, we have
$$(cN\tau+d)^kf(N\tau)=f\left(\frac{aN\tau+b}{cN\tau+d}\right)=f\left(N\frac{a\tau+b/N}{Nc\tau+d}\right).$$
\end{proof}

\section{Eisenstein series}

\subsection*{Level $1$}

Taking the logarithmic derivative of the product formula of the sine, we obtain 
$$\pi\cot\pi z=i\pi(1-2\sum_{k=0}^{\infty}e^{2\pi ikz})=\frac{1}{z}+\sum_{n=1}^{\infty}\left({\frac{1}{z+n}+\frac{1}{z-n}}\right),$$
assuming that $z\in\mathfrak{h}$. Successive differentiation w.r.t.~$z$ yields (for $r>1$)
$$\frac{1}{z^r}+\sum_{n=1}^{\infty}\left({\frac{1}{(z+n)^r}+\frac{1}{(z-n)^r}}\right)=\frac{(-2\pi i)^r}{(r-1)!}\sum_{k=1}^{\infty}{k^{r-1}e^{2\pi ikz}}.$$
We may now define the {\em Eisenstein series}: Let $k>2$ be an even integer; then
$$E_{k}(\tau)=\zeta(k)^{-1}{\sum_{m,n}}^{\prime}(m\tau+n)^{-k}=1+2\frac{(-2\pi i)^k}{\zeta(k)(k-1)!}\sum_{m=1}^{\infty}\sum_{d=1}^{\infty}d^{k-1}q^{dm},$$
where the prime denotes the omission of the term $m=n=0$. 
This in turn implies
$$E_k(\tau)=1-\frac{2k}{B_k}\sum_{n=1}^{\infty}{\sigma_{k-1}(n)q^n},$$
where $\sigma_k(n)=\sum_{d|n}d^k$. It is straightforward to see that $E_k(\tau+1)=E_k(\tau)$ and $\tau^kE_k(\tau)=E_k(-1/\tau)$.
For $k=2$, we define
$$E_2=1-24\sum_{n=1}^{\infty}\sigma_1(n)q^n,$$
which is holomorphic, but not modular anymore; instead, we have
$$\tau^{-2}E_2(-1/\tau)=E_2+\frac{6}{i\pi\tau},$$
which means that
$$\hat{E}_2=E_2-\frac{3}{\pi\Im({\tau})}$$
behaves well w.r.t.~$\Gamma=SL(2,\mathbb{Z})$.  Sometimes it will be convenient to use another normalization of the Eisenstein series, namely
$$G_k(\tau)=-\frac{B_k}{2k}E_k(\tau).$$

\subsection*{Level $N$}

Fixing a level $N>1$, we may choose to sum over a sublattice $(m_1,m_2)\equiv(a_1,a_2)\mod N$, so, for any integer $k\geq3$, we define

$${\mathcal{G}}_{k}^{(a_1,a_2)}(\tau)=\sum_{{\bf{m}}\equiv{\bf{a}}\ (N)}(m_1\tau+m_2)^{-k}.$$
Obviously, we have
$$(c\tau+d)^{-k}\mathcal{G}_{k}^{(a_1,a_2)}\left(\frac{a\tau+b}{c\tau+d}\right)=\sum_{{\bf{m}}\equiv{\bf{a}}\ (N)}((m_1a+m_2c)\tau+(m_1b+m_2d))^{-k},$$
and we may change the summation to run over the lattice $\bf{m}^{\prime}=\bf{m}\gamma\equiv\bf{a}\gamma$; take note that, modulo $N$,  $\Gamma_1(N)$ preserves ${\bf{a}}=(0,a_2)$. 
In order to obtain a $q$-expansion, we recast the expression for $\mathcal{G}$ by splitting the sums
\begin{equation*}
\begin{split}{\mathcal{G}}_{k}^{(a_1,a_2)}(\tau)&=b_k+N^{-k}\sum_{0<m_1\equiv a_1}\sum_{n\in\mathbb{Z}}\left(\frac{m_1\tau+a_2}{N}+n\right)^{-k}\\
 &\quad+(-N)^{-k}\sum_{0<m_1\equiv-a_1}\sum_{n\in\mathbb{Z}}\left(\frac{m_1\tau-a_2}{N}+n\right)^{-k}\\
 &=b_k+c_k\sum_{0<m_1\equiv a_1}\sum_{d=1}^{\infty}d^{k-1}e^{2\pi i d (m_1\tau+a_2)/N}\\
&\quad+(-1)^kc_k\sum_{0<m_1\equiv -a_1}\sum_{d=1}^{\infty}d^{k-1}e^{2\pi i d (m_1\tau-a_2)/N},\\
\end{split}
\end{equation*}
where $c_k=(-2\pi i/N)^k/((k-1)!)$ and $b_k=\sum_{n\equiv a_2}n^{-k}$.

\subsection*{Level $3$}

Let us focus on $N=3$ and ${\bf{a}}\equiv(0,1)$. For  $k=2n+1$, we have
\begin{equation*}
\begin{split}
b_{2n+1}&=3^{-k}\zeta(2n+1,1/3)-3^{-k}\zeta(2n+1,2/3)\\
&=(-1)^n\frac{(2\pi)^{2n+1}}{(2n)!\sqrt{3}}\zeta(-2n,1/3),\\
\end{split}
\end{equation*}
and
\begin{equation*}\begin{split}
&(-1)^n\frac{(2\pi)^{2n+1}}{3^{2n+1}(2n)!}i^{-1}\sum_{l=1}^{\infty}\sum_{d=1}^{\infty}d^{2n}(e^{2\pi id/3}-e^{-2\pi id/3})q^{dl}\\
=&(-1)^n\frac{(2\pi)^{2n+1}}{3^{2n+1}(2n)!}\sqrt{3}\sum_{l=1}^{\infty}\left(\sum_{d|l}d^{2n}(\frac{d}{3})\right)q^l,\\
\end{split}
\end{equation*}
where we introduced the Legendre symbol,
$$(\frac{d}{3})=-1,0,1 \mbox{ for } d\equiv-1,0,1\mod3.$$
Thus, the normalized odd Eisenstein series for $\Gamma_1(3)$ is given by 
$$E_{2n+1}^{\Gamma_1(3)}(\tau)=1-\frac{2n+1}{3^{2n}B_{2n+1}(1/3)}\sum_{l=1}^{\infty}\left(\sum_{d|l}d^{2n}(\frac{d}{3})\right)q^l.$$
Similarly, for even $k=2n+2$ we compute
\begin{equation*}
\begin{split}
\frac{b_k}{c_k}&=(-1)^{n+1}\frac{(2n+1)!}{(2\pi)^{2n+2}}\left(\zeta(2n+2,1/3)+\zeta(2n+2,2/3)\right)\\
&=\zeta(-(2n+1),1)-{3^{2n+2}}\zeta(-(2n+1),1/3)\\
&=-\frac{B_{2n+2}}{2n+2}+3^{2n+2}\frac{B_{2n+2}(1/3)}{2n+2}\\
&=-\frac{B_{2n+2}}{2n+2}+3^{2n+2}\frac{B_{2n+2}(1/3)+B_{2n+2}(2/3)}{4n+4}\\
&=\left(-1+\frac{3-3^{2n+2}}{2}\right)\frac{B_{2n+2}}{2n+2}
\end{split}
\end{equation*}
which implies that for $n\geq1$
$$2G_{2n+2}^{(0,1)(3)}(\tau):=-\frac{1}{2}(3^{2n+2}-1)\frac{B_{2n+2}}{2n+2}+\sum_{l=1}^{\infty}\sum_{d=1}^{\infty}d^{2n+1}(e^{2\pi i d/3}+e^{-2\pi i d/3})q^{dl}$$
is a modular form  w.r.t.~$\Gamma=\Gamma_1(3)$. In order to treat the situation $n=0$, we observe that$$2G_{2}^{(0,1)(3)}(\tau)-2G_2(\tau)=\frac{B_2}{2}(1-\frac{9}{2}+\frac{1}{2})+\sum_{l=1}^{\infty}\sum_{d=1}^{\infty}d(e^{2\pi i d/3}+e^{-2\pi i d/3}-2)q^{dl}$$
$$=-3\left(\frac{B_2}{2}+\sum_{n=1}^{\infty}\left(\sum_{3\nmid d|n}d\right)q^n\right)=-3\left(\frac{1}{12}+\sum_{n=1}^{\infty}(\sigma_1(n)-3\sigma_1(n/3))q^n\right);$$
this  is proportional to $E_2(\tau)-3E_2(3\tau)=\hat{E}_2(\tau)-3\hat{E}_2(3\tau)$, hence even modular for $\Gamma_0(3)$.\newline
\newline
It is also possible to define a first Eisenstein series by introducing a regularization scheme that preserves modularity (cf.~e.g.~\cite{schoeneberg}); we still focus on the level $N=3$ and define 
$${\mathcal{G}}_{1}^{(a_1,a_2)}(\tau,s)={\sum_{{\bf{m}}\equiv{\bf{a}}}}^{\prime}(m_1\tau+m_2)^{-1}|m_1\tau+m_2|^{-s},$$
which converges for $s$ sufficiently large, so we rearrange the sums as
$${\mathcal{G}}_{1}^{(a_1,a_2)}(\tau,s)=b(s,{\bf{a}})+{\sum_{m_1\equiv a_1}}^{\prime}\sum_{n\in\mathbb{Z}}3^{-1-s}\left(\frac{m_1\tau+a_2}{3}+n\right)^{-1}\left|\frac{m_1\tau+a_2}{3}+n\right|^{-s}$$
using 
$$b(s,{\bf{a}})=\left\{ \begin{array}{ll}\sum_{m_2\equiv a_2}^{\prime}m_2^{-1}|m_2|^{-s}& \mbox{if } a_1\equiv0\mod3\\ 0& \mbox{otherwise}  \end{array} \right. .$$
Next, we observe that for fixed $z$
$$\Psi(u,s)=\sum_{k\in\mathbb{Z}}(z+k+u)[(z+k+u)(\bar z+k+u)]^{-s/2}$$
may be Fourier transformed into
$$\Psi(u,s)=\sum_{m\in\mathbb{Z}}c_m(z,s)e^{2\pi i mu};$$
we may evaluate the Fourier coefficients
\begin{equation*}\begin{split} c_m(z,s)&=\int_0^1\sum_{k\in\mathbb{Z}}(z+k+u)^{-1}|z+k+u|^{-s}e^{-2\pi imu}du\\  &=\int_{-\infty}^{\infty}(z+u)^{-1}|z+u|^{-s}e^{-2\pi imu}du\ ,\\
\end{split}
\end{equation*}
so we have
$$c_m\left(\frac{m_1\tau+a_2}{3},0\right)=\left\{\begin{array}{rcl}-2\pi i\mbox{sgn}(m)e^{2\pi im(m_1\tau+a_2)/3}&\mbox{for}&mm_1>0 \\ 0&\mbox{for}&mm_1<0\end{array} \right.;$$
the zero mode enters into the Eisenstein series via
$$3^{-1-s}{\sum_{m_1\equiv a_1}}^{\prime}c_{0}(z,s)=3^{-1}{\sum_{m_1\equiv a_1}}^{\prime}\ \frac{\mbox{sgn}(m_1)}{|m_1|^s}\int_{-\infty}^{\infty}(\tau+u)^{-1}|\tau+u|^{-s}du,$$
which is holomorphic at $s=0$, and, in particular, does not contribute for $a_1\equiv0$. Finally, we evaluate $b$ by considering
$$ \lim_{s\rightarrow1}[\zeta(s,x)-1/(s-1)]=-\Gamma^{\prime}(x)/\Gamma(x)=-\psi(x),$$
where the digamma function $\psi$ satisfies
$$\psi(1-x)-\psi(x)=\pi\cot(\pi x),$$
which implies that
$$b(0,(0,1))=3^{-1}\lim_{s\rightarrow0}[\zeta(1+s,1/3)-\zeta(1+s,2/3)]=\frac{\pi}{3\sqrt{3}}.$$
Putting everything together, we see that the regularized (and normalized) Eisenstein series has the same $q$-expansion that would have been expected from the na\"ive formula:
$$E^{\Gamma_1(3)}_1=1+6\sum_{n=1}^{\infty}\left(\sum_{d|n}(\frac{d}{3})\right)q^n.$$
Finally, from \eqref{valence}, we know that
$$\dim M_k\left(\Gamma_1(3)\right)\leq1+\frac{k}{3};$$
thus, it is easy to see that the ring of modular forms w.r.t.~$\Gamma_1(3)$ is generated by the first and third Eisenstein series; as confusion is unlikely, we denote them by $E_1$ and $E_3$, respectively.

\section{Expanding the Hirzebruch genus}\label{expell}

Following \cite{HBJ}, we introduce the  $\Phi$-function, 
\begin{equation*}
\begin{split}
\Phi(\tau,x)&=(\xi^{1/2}-\xi^{-1/2})\prod_{n=1}^{\infty}\frac{(1-\xi q^n)(1-\xi^{-1}q^n)}{(1-q^n)^2}\\
&=x\exp\left(-\sum_{k=1}^{\infty}\frac{2}{(2k)!}G_{2k}(\tau)x^{2k}\right),\\
\end{split}
\end{equation*}
where $\xi=\exp x$ and $q=\exp(2\pi i \tau)$; for ${{\left(\begin{array}{cc} a & b \\ c & d \end{array}\right) \in SL(2,\mathbb{Z})}}$,  
it satisfies
$$\Phi\left(\frac{a\tau+b}{c\tau+d},\frac{x}{c\tau+d}\right)(c\tau+d)=\exp\left(\frac{cx^2}{4\pi i (c\tau+d)}\right)\Phi(\tau,x).$$
Integrating and exponentiating 
$$\partial_x\ln\Phi(\tau,x)=\frac{1}{2}\coth\frac{x}{2}-\sum_{n=1}^{\infty}\sum_{d|n}(e^{dx}-e^{-dx})q^n,$$
we conclude that
$$\Phi(\tau,x)=\exp\left(-2\sum_{k=1}^{\infty}\frac{(x-\omega)^k}{k!}G_k^{(\omega)}(\tau)\right)\cdot\left\{\begin{array}{rcl} x & \mbox{for} & \omega=0 \\ \Phi(\tau,\omega)& \mbox{for} & \omega \neq 0 \\ \end{array} \right. ,$$
where
$$G_k^{(\omega)}(\tau)=-\frac{c_k^{(\omega)}}{2k}+\sum_{n=1}^{\infty}\left(\sum_{d|n}\frac{e^{d\omega}+(-1)^ke^{-d\omega}}{2}d^{k-1}\right)q^n$$
and 
$$\frac{1}{2}\coth\frac{x}{2}=\sum_{k=0}^{\infty}c_k^{(\omega)}\frac{(x-\omega)^{k-1}}{k!}.$$
Therefore we have:
\begin{equation*}\begin{split}
&\frac{x}{1-e^{-x}}\frac{1-e^ae^{-x}}{1-e^a}\prod_{n=1}^{\infty}\frac{(1-q^n)^2}{(1-q^ne^x)(1-q^ne^{-x})}\frac{1-e^{-a}q^ne^x}{1-e^{-a}q^n}\frac{1-e^{a}q^ne^{-x}}{1-e^{a}q^n}\\
&=x\frac{\Phi(\tau,x-a)}{\Phi(\tau,x)\Phi(\tau,-a)}=\exp\left(2\sum_{n=1}^{\infty}\frac{x^{2n}}{(2n)!}G_{2n}(\tau)-2\sum_{k=1}^{\infty}\frac{x^k}{k!}G_k^{(-a)}(\tau)\right).\\
\end{split}
\end{equation*}

\subsection*{Deriving the $c_k$ at level three}

By definition,
$$\frac{1}{2}\frac{e^x+e^{\omega}}{e^x-e^{\omega}}=\frac{1}{2}\coth\left(\frac{x-\omega}{2}\right)=\sum_{k=0}^{\infty}\frac{c_k^{(-\omega)}}{k!}(x-\omega-(-\omega))^{k-1}.$$
We set $e^{\omega}=\zeta\neq1$, and, at level three, $\zeta^3=1$. Then,
\begin{equation*}
\begin{split}
\frac{e^x+\zeta}{e^x-\zeta}&=(e^{3x}-1)^{-1}\left((\zeta-\zeta^2)(e^{2x}-e^x)+e^{3x}-e^{2x}-e^x+1\right) \\ &=(\zeta-\zeta^2)\sum_{k=-1}^{\infty}(B_{k+1}({2}/{3})-B_{k+1}({1}/{3}))\frac{(3x)^k}{(k+1)!} \\  &\quad +\sum_{k=-1}^{\infty}(B_{k+1}(1)+B_{k+1}(0)-B_{k+1}({2}/{3})-B_{k+1}({1}/{3}))\frac{(3x)^k}{(k+1)!}\\
&=-2(\zeta-\zeta^2)\sum_{k=0}^{\infty}3^{2k}B_{2k+1}(1/3)\frac{x^{2k}}{(2k+1)!}\\ &\quad+\sum_{n=0}^{\infty}(3^{2n+2}-1)B_{2n+2}\frac{x^{2n+1}}{(2n+2)!}.\\
\end{split}
\end{equation*}
The last step follows from the fact that the minus first summands cancel, and we used
$$B_n(1)=3^{n-1}(B_n(1/3)+B_n(2/3)+B_n(1)),$$ $$B_n(1-x)=(-1)^nB_n(x).$$
So we have
\begin{equation*}
\begin{split}
c_{2n+1}^{(-\omega)}&=(e^{-\omega}-e^{\omega})3^{2n}B_{2n+1}(1/3),\\
c_{2n+2}^{(-\omega)}&=\frac{3^{2n+2}-1}{2}B_{2n+2}.\\
\end{split}
\end{equation*}
Thus, we may express the  elliptic genus of level three as
$$x\frac{\Phi(\tau,x-\omega)}{\Phi(\tau,x)\Phi(\tau,-\omega)}=\exp\left(3\sum_{n=1}^{\infty}\frac{x^{2n}}{(2n)!}G_{2n}^*(\tau)-2\sum_{k=0}^{\infty}\frac{x^{2k+1}}{(2k+1)!}G_{2k+1}^{(-\omega)}(\tau)\right),$$
where
\begin{equation*}
\begin{split}
G_{2n}^*(\tau)&=G_{2n}(\tau)-3^{2n-1}G_{2n}(3\tau),\\
G_{2k+1}^{(-\omega)}(\tau)&=\frac{e^{\omega}-e^{-\omega}}{2}3^{2k}\frac{B_{2k+1}(1/3)}{2k+1}E_{2k+1}^{\Gamma_1(3)}(\tau),\\
\end{split}
\end{equation*}
hence modularity is manifest. 
Choosing $\omega=2\pi i/3$, the first few terms of the genus, when expressed in terms of $E_1$ and $E_3$, read
\begin{equation}\label{expandedgenus}
\begin{split}
Ell^{\Gamma_1(3)}(x) & = 1+\frac{iE_1}{2\sqrt{3}}x+\frac{E_1^2}{12}x^2+\frac{iE_1^3-iE_3}{18\sqrt{3}}x^3+\frac{13E_1^4-16E_1E_3}{2160}x^4\\ &\quad+\frac{iE_1^2(E_1^3-E_3)}{216\sqrt{3}}x^5+\frac{121E_1^6-152E_1^3E_3+40E_3^2}{272160}x^6+O(x^7).\\
\end{split}
\end{equation}
We also list the first few terms of the expansion of the Todd genus:
\begin{equation}\label{expandedtodd}
\frac{x}{1-e^{-x}}=1+\frac{1}{2}x+\frac{1}{12}x^2-\frac{1}{720}x^4+\frac{1}{30240}x^6+O(x^8).\end{equation}

\section{Useful congruences}\label{usecong}

The ring of modular forms w.r.t.~$\Gamma=\Gamma_1(3)$ is generated by
\begin{equation*}
\begin{split}
E_1&=1+6\sum_{n=1}^{\infty}\left(\sum_{d|n}(\frac{d}{3})\right)q^n,\\
E_3&=1-9\sum_{n=1}^{\infty}\left(\sum_{d|n}(\frac{d}{3})d^2\right)q^n,\\
\end{split}
\end{equation*}
so it is straightforward to check that
\begin{equation*}
\begin{split}
E_4&=1+240\sum_{n=1}^{\infty}\left(\sum_{d|n}d^3\right)q^n=9E_1^4-8E_1E_3,\\
E_6&=1-504\sum_{n=1}^{\infty}\left(\sum_{d|n}d^5\right)q^n=-27E_1^6+36E_1^3E_3-8E_3^2,\\
E_8&=1+480\sum_{n=1}^{\infty}\left(\sum_{d|n}d^7\right)q^n=E_4^2,\\
\end{split}
\end{equation*}
\begin{equation*}
\begin{split}
G_2^*&=\frac{1}{12}+\sum_{n=1}^{\infty}\left(\sum_{3\nmid d|n}d\right)q^n=\frac{1}{12}E_1^2,\\
G_4^*&=-\frac{13}{120}+\sum_{n=1}^{\infty}\left(\sum_{3\nmid d|n}d^3\right)q^n=\frac{1}{40}E_1^4-\frac{2}{15}E_1E_3,\\
G_6^*&=\frac{121}{252}+\sum_{n=1}^{\infty}\left(\sum_{3\nmid d|n}d^5\right)q^n=\frac{1}{28}E_1^6+\frac{2}{7}E_1^3E_3+\frac{10}{63}E_3^2,\\
G_8^*&=-\frac{1093}{240}+\sum_{n=1}^{\infty}\left(\sum_{3\nmid d|n}d^7\right)q^n=\frac{9}{80}E_1^8-\frac{6}{5}E_1^5E_3-\frac{52}{15}E_1^2E_3^2,\\
\end{split}
\end{equation*}
from which we may derive `obvious' congruences, e.g.
$$\frac{E_1-1}{6}\in \mathbb{Z}[\![q]\!].$$
Furthermore, we have
\begin{lem}\label{E_4power}
Let $l=(2n+1)2^m$. Then
$$\frac{E_4^l-1}{2^{4+m}}\in\mathbb{Z}[\![q]\!].$$
\end{lem}
\begin{proof}
We rewrite
$$\left(\begin{array}{c} l\\ i\end{array}\right)=\frac{2^m(2n+1)(l-1)!/(l-i)!}{i!},$$
and factorize the powers of two using Legendre's formula
$$\nu_2(i!)=\sum_{k\geq1}\left\lfloor\frac{i}{2^k}\right\rfloor<\sum_{k\geq1}\frac{i}{2^k}=i,$$
which implies
\begin{equation}\label{ichoosel}
2^m\ |\left(\begin{array}{c} l\\ i\end{array}\right)2^{i-1}\ \mbox{for}\ i\geq1.
\end{equation}
Writing  $E_4=1+16Q$, i.e.~
\begin{equation}\label{chooseE_4}
E_4^l-1=\sum_{i=1}^l\left(\begin{array}{c} l \\ i \end{array}\right)2^{4i}Q^i,
\end{equation}
the claim follows.
\end{proof}

But there also more subtle congruences, two of which we list in the following Propositions:

\begin{prop}\label{E_1squaredE_3}
$$\frac{1}{2}\left\{\frac{E_1^2-1}{12}+(2k+1)\frac{E_3-1}{9}\right\}\in\mathbb{Z}[\![q]\!].$$
\end{prop}
\begin{proof}
We have to show that
$$\frac{1}{2}\left(\sum_{n=1}^{\infty}\left(\sum_{3\nmid d|n}d\right)q^n-(2k+1)\sum_{n=1}^{\infty}\left(\sum_{d|n}(\frac{d}{3})d^2\right)q^n\right)\in\mathbb{Z}[\![q]\!],$$
but obviously
\begin{equation*}
\sum_{d|n}(\frac{d}{3})d^2\equiv\sum_{3\nmid d|n}d^2\equiv\sum_{3\nmid d|n}d \mod2.
\end{equation*}
\end{proof}

\begin{prop}\label{E_4/64}
$$\frac{1}{4}\frac{E_4-1}{16}-\frac{1}{8}\frac{E_1^2-1}{4}+\frac{1}{4}(E_3^2-1)+\frac{1}{8}(E_1^3E_3-1)\in\mathbb{Z}[\![q]\!].$$
\end{prop}
\begin{proof}
We add an obvious congruence to the square of a term that expands integrally by Proposition \ref{E_1squaredE_3} and compute  
\begin{equation*}
\begin{split}0&\equiv\left(\frac{E_1^2-1}{8}+\frac{E_3-1}{2}\right)^2+\frac{E_1-1}{2}\frac{E_1^2-1}{4}E_3\\
&=\frac{E_1^4-1}{2^6}-\frac{E_1^2-1}{2^5}+\frac{1}{2}\frac{E_1^2-1}{2^2}(E_3-1)+\frac{1}{2^2}(E_3^2-1)-\frac{1}{2}(E_3-1)\\ &\quad+\frac{E_1-1}{2}\frac{E_1^2-1}{4}E_3\\
&=\frac{E_1^4-1}{2^6}-\frac{E_1^2-1}{2^5}+\frac{1}{2^2}(E_3^2-1)+\frac{1}{2}\frac{E_1^2-1}{4}E_1E_3-\left(\frac{E_1^2-1}{2^3}+\frac{E_3-1}{2}\right)\\ 
&\equiv\frac{E_1^4-1-8E_1E_3}{2^6}-\frac{E_1^2-1}{2^5}+\frac{1}{2^2}(E_3^2-1)+\frac{1}{2^3}E_1^3E_3\\
&=\frac{E_4-1}{2^6}-\frac{E_1^4-1}{2^3}-\frac{E_1^2-1}{2^5}+\frac{1}{2^2}(E_3^2-1)+\frac{1}{2^3}(E_1^3E_3-1),\\
\end{split}
\end{equation*}
but $(E_1^4-1)/8\in\mathbb{Z}[\![q]\!]$.
\end{proof}

\begin{lem}\label{E_1squaredvanishes}
$$\frac{1}{8}\frac{E_1^2-1}{4}\equiv0 \mod \underline{\underline{D}}_8^{\Gamma}.$$
\end{lem}
\begin{proof}
By adding modular forms of weight 8 we obtain the following relations:
$$\frac{1}{8}E_1^3E_3\equiv-\left(\left(\frac{E_1^2-1}{4}\right)^2+\frac{E_4-1}{16}\right)E_1E_3\equiv0 \mod \underline{\underline{D}}_8^{\Gamma},$$
$$\frac{1}{4}E_3^2\equiv-\frac{E_1^2-1}{4}E_3^2\equiv0 \mod \underline{\underline{D}}_8^{\Gamma},$$
$$\frac{1}{4}\frac{E_4-1}{16}\equiv-2\left(\frac{E_4-1}{16}\right)^2\equiv0 \mod \underline{\underline{D}}_8^{\Gamma};$$
these yield the claim upon insertion into \ref{E_4/64}.
\end{proof}

\begin{prop}\label{g2_ugly_2}
$$2\frac{121E_1^6-152E_1^3E_3+40E_3^2}{2^7\cdot3^6\cdot5\cdot7}+18\frac{E_1^2}{2^7\cdot3^4\cdot5\cdot7}\equiv\frac{1}{2}\left(\frac{E_1^2-1}{4}\right)^3\mod \underline{\underline{D}}_8^{\Gamma}.$$
\end{prop}
\begin{proof}
We start by removing the unwanted primes from the denominator:
\begin{equation*}
\begin{split}
&\ \quad2\frac{121E_1^6-152E_1^3E_3+40E_3^2}{2^7\cdot3^6\cdot5\cdot7}+18\frac{E_1^2}{2^7\cdot3^4\cdot5\cdot7}\\
&=\frac{28E_1^3E_3-14E_1^6-5E_6}{2^6\cdot3^6\cdot5\cdot7}+\frac{81E_1^2}{2^6\cdot3^6\cdot5\cdot7}\\
&=\frac{4E_1^3E_3-2E_1^6}{2^6\cdot3^6\cdot5}-\frac{E_6}{2^6\cdot3^6\cdot7}+\frac{81E_1^2}{2^6\cdot3^6\cdot5\cdot7}\\
&\equiv\frac{4E_1^3E_3-2E_1^6}{2^6\cdot3^6\cdot5}-\frac{55E_6}{2^6\cdot3^6\cdot7}-\frac{189E_1^2}{2^6\cdot3^6\cdot5\cdot7}\\
&=\frac{4E_1^3E_3-2E_1^6-22E_1^2}{2^6\cdot3^6\cdot5}-\frac{E_1^2+E_6}{2^6\cdot3^6}-\frac{48E_6}{2^6\cdot3^6\cdot7}\\
&\equiv\frac{4E_1^3E_3-2E_1^6-22E_1^2}{2^6\cdot3^6\cdot5}-\frac{E_1^2+E_6}{2^6\cdot3^6}\\
&=\frac{E_1^3E_3}{2^4\cdot3^6}-\frac{E_1^6+E_1^2}{2^4\cdot3^6}+\frac{1}{2\cdot3^6}\frac{E_4-1}{80}E_1^2-\frac{E_1^2+E_6}{2^6\cdot3^6}\\
&\equiv\frac{E_1^3E_3}{2^4\cdot3^6}-\frac{E_1^2+E_6}{2^6\cdot3^6},\\
\end{split}
\end{equation*}
by using
$$-\frac{E_1^2+E_6}{2^5\cdot3^3\cdot7}\equiv\frac{E_1^2-1}{2^2\cdot3}\frac{E_6-1}{2^3\cdot3^2\cdot7}\equiv0,$$
$$48\frac{E_6-1}{2^6\cdot3^6\cdot7}\in\mathbb{Z}[1/3][\![q]\!],$$
$$\frac{2E_1^3E_3-E_1^6-11E_1^2}{2^5\cdot3^6\cdot5}=\frac{10(E_1^3E_3-E_1^6-E_1^2)+E_1^2(-8E_1E_3+9E_1^4-1)}{2^5\cdot3^6\cdot5},$$
$$\frac{1}{2\cdot3^6}\frac{E_4-1}{80}E_1^2\equiv\frac{1}{2}\frac{E_4-1}{16}E_1^2\equiv-\frac{1}{2}\frac{E_4-1}{16}\equiv4\left(\frac{E_4-1}{16}\right)^2\equiv0,$$ 
$$-\frac{E_1^6+E_1^2}{2^4\cdot3^3}\equiv\frac{E_1^6-1}{2^2\cdot3^2}\frac{E_1^2-1}{2^2\cdot3}\equiv0.$$
So we have
\begin{equation*}
\begin{split}
&\ 2\frac{121E_1^6-152E_1^3E_3+40E_3^2}{2^7\cdot3^6\cdot5\cdot7}+18\frac{E_1^2}{2^7\cdot3^4\cdot5\cdot7}\\
\equiv&\ \frac{1}{2\cdot3^6}\left\{\frac{E_1^2-1}{4}\frac{E_6-1}{8}-\left(\frac{E_1^2-1}{4}\right)^2E_1E_3-\frac{E_4-1}{16}E_1E_3\right\}\\
\equiv&\ \frac{1}{2}\left\{\frac{E_1^2-1}{4}\frac{E_6-1}{8}-\left(\frac{E_1^2-1}{4}\right)^2E_1E_3-\frac{E_4-1}{16}E_1E_3\right\},\\
\end{split}
\end{equation*}
which is manifestly of order two.
Dropping top and zero weight forms and writing $E_6$ in terms of $E_1$, $E_3$, this reads
$$\frac{1}{2}\left\{-\frac{E_1^2}{32}-\frac{36E_1^3E_3-27E_1^6-8E_3^2}{32}+\frac{E_1^3E_3}{8}\right\}=\frac{27}{64}E_1^6+\frac{1}{8}E_3^2-\frac{1}{2}E_1^3E_3-\frac{1}{64}E_1^2.$$
Clearly, we may also drop $E_1^3E_3/2\equiv E_1^5E_3/2$; furthermore, we have
\begin{equation*}
\begin{split}
\frac{E_3^2}{8}&\equiv\frac{1}{2}\frac{1-E_1^2}{4}E_3^2=\frac{1-E_1^2}{4}\left\{\frac{1}{2}+E_3-1+\frac{1}{2}(E_3-1)^2\right\}\\
&\equiv\frac{E_1^2-1}{4}\left\{\frac{1}{2}+\frac{1}{2}(E_3-1)^2\right\}\equiv\frac{1}{2}\left\{\frac{E_1^2-1}{4}+\left(\frac{E_1^2-1}{4}\right)^3\right\}.\\
\end{split}
\end{equation*} 
Finally, we use $E_1^6/64\equiv E_1^2/64$, derived from considering 
$$0\equiv\left(\frac{E_1^2-1}{4}\right)^2\frac{E_1^4-1}{8}=\frac{E_1^8-2E_1^6+2E_1^2-1}{128},$$ and invoke lemma \ref{E_1squaredvanishes}.
\end{proof}

\section{Derivation of the $\hat{\eta}$-form}\label{etazhang}

We use the notation from the proof of Proposition \ref{eform} and follow \cite{zhang94}. The superconnection associated to the family of $Spin^{\mathbb{C}}$ Dirac operators on the circle bundle $S(L)$ is given by
$$\mathbb{A}_t=\tilde\nabla+\sqrt{t}\eth-\frac{c(T)}{4\sqrt{t}}.$$
Introducing a Grassmann variable $z$, we may write down a Weitzenb\"ock-type formula
$$z\left(\sqrt{t}\eth+\frac{c(T)}{4\sqrt{t}}\right)-\mathbb{A}_{t}^2=t\left(\nabla_e+\frac{T}{4t}+z\frac{c(e)}{2\sqrt{t}}\right)^2,$$
where we identified $T$ on the RHS with a basic two-form (by contraction with the dual of the vertical unit tangent $e$).
Now we have to compute
$${\mbox{Tr}}^{ev}\left[\frac{d\mathbb{A}_t}{dt}\exp(-\mathbb{A}_t^2)\right]=\frac{1}{2t}\mbox{Tr}^z\left[\exp\left(t\left(\nabla_e+\frac{T}{4t}+z\frac{c(e)}{2\sqrt{t}}\right)^2\right)\right],$$
where $\mbox{Tr}^z$ denotes the trace restricted to the coefficient of $z$. We insert $\mbox{Tr}[c(e)]=-i$, use  a Fourier decomposition of the sections of the spinor bundle, and perform Poisson resummation to obtain
\begin{equation*}
\begin{split}
 &\ \quad t^{-1}\mbox{Tr}^z\left[\exp\left(t\left(\nabla_e+\frac{T}{4t}+z\frac{c(e)}{2\sqrt{t}}\right)^2\right)\right]\\
&=t^{-1}\left[\sum_{k\in\mathbb{Z}}\exp\left(t\left(ik+\frac{T}{4t}-z\frac{i}{2\sqrt{t}}\right)^2\right)\right]^z\\
&=t^{-1/2}\left[\sum_{k\in\mathbb{Z}}\left(k-i\frac{T}{4{t}}\right)\exp\left(-t\left(k-i\frac{T}{4t}\right)^2\right)\right]\\
&=-\frac{i\pi^{3/2}}{t^2}\sum_{k\in\mathbb{Z}}k\exp\left(2\pi kT/(4t)-k^2\pi^2/t\right).\\
\end{split}
\end{equation*}
Since $\exp(aT)$ is to be understood as formal power series, we may integrate for $k\neq0$
\begin{equation*}
\begin{split}
&\ \quad k\int_0^{\infty}\frac{1}{t^2}\sum_{l=0}^{\infty}\frac{(\pi kT/(2t))^l}{l!}\exp\left(-\frac{k^2\pi^2}{t}\right)dt\\
&=k\int_0^{\infty}\sum_{l=0}^{\infty}\frac{(\pi kTx/2)^l}{l!}\exp\left(-k^2\pi^2x\right)dx\\
&=\frac{1}{k\pi^2}\sum_{l=0}^{\infty}\left(\frac{T}{2\pi k}\right)^l=\frac{1}{\pi^2}\frac{1}{k-\frac{T}{2\pi}}.\\
\end{split}
\end{equation*}
Thus, putting everything together, we have
$$2\tilde{\eta}=\frac{i}{\pi}\sum_{k\neq0}\frac{1}{\frac{T}{2\pi}-k}=\frac{i}{\pi}\left(\pi\cot\left(\frac{T}{2}\right)-\frac{2\pi}{T} \right)=\left(\coth\left(\frac{T}{2i}\right)-\frac{2i}{T}\right).$$
Obviously, we may identify the basic two-form with minus $i$ times the curvature two-form of the hermitian line, i.e. $T=-iF$; therefore, the normalized $\hat{\eta}$-form is given by
$$\hat{\eta}(\eth)=\sum_{k=1}^{\infty}\frac{B_{k+1}}{(k+1)!}\left(\frac{iF}{2\pi}\right)^k.$$

\addcontentsline{toc}{section}{References}

\bibliography{on_the_geometry_of_the_f-invariant_v2}

\end{document}